\documentclass[a4paper]{amsart}
\UseRawInputEncoding
\usepackage{amsmath}
\usepackage{bm}
\usepackage{geometry,amsthm,graphics,tabularx,amssymb,shapepar}
\usepackage{amscd}
\usepackage{color}
\usepackage{mathrsfs}
\usepackage{comment}
\usepackage[all]{xypic}
\allowdisplaybreaks
\newcommand{\CC}{{\mathcal{C}}}
\newcommand{\bfk}{{\bf  k}}
\newcommand{\bfd}{{\bf d}}

\newcommand{\mv}{\mathsf{Vir}}
\newcommand{\td}{\tilde{\rd}}
\newcommand{\fp}{\mathfrak p}
\newcommand{\fb}{\mathfrak b}

\newcommand{\ft}{\mathfrak t}
\newcommand{\ff}{\mathfrak f}

\newcommand{\fg}{\mathfrak g}
\newcommand{\fh}{\mathfrak h}

\newcommand{\wt}{\widetilde}
\newcommand{\wh}{\widehat}

\newcommand{\Der}{\mathrm{Der}}

\newcommand{\geqs}{\geqslant}
\newcommand{\leqs}{\leqslant}

\newcommand{\ot}{\otimes}

\newcommand{\zdwz}{z^{\e-1}\delta\(\frac{w}{z}\)}
\newcommand{\wpw}{w^\e\frac{\partial}{\partial w}}

\newcommand{\C}{\mathbb{C}}

\newcommand{\N}{\mathbb N}
\newcommand{\Z}{\mathbb Z}
\newcommand{\rK}{\mathrm{K}}
\newcommand{\rD}{\mathrm{D}}
\newcommand{\rd}{\mathrm{d}}

\newcommand{\p}{\partial}
\newcommand{\e}{\epsilon}
\newcommand{\rk}{\mathrm{k}}

\newcommand{\ba}{\begin {eqnarray}}
\newcommand{\ea}{\end {eqnarray}}
\newcommand{\baa}{\begin {eqnarray*}}
\newcommand{\eaa}{\end {eqnarray*}}
\newcommand{\be}{\begin {equation}}
\newcommand{\ee}{\end {equation}}
\newcommand{\bee}{\begin {equation*}}
\newcommand{\eee}{\end {equation*}}

\newcommand{\U}{\mathcal{U}}

\newcommand{\te}[1]{\textnormal{{#1}}}

\theoremstyle{Theorem}

\theoremstyle{Theorem}

\newtheorem{thm}{Theorem}[section]

\newtheorem{lemt}[thm]{Lemma}
\newtheorem{prpt}[thm]{Proposition}
\newtheorem{thmt}[thm]{Theorem}
\newtheorem{remt}[thm]{Remark}

\theoremstyle{Theorem}

\theoremstyle{Theorem}

\theoremstyle{Plain}

\theoremstyle{Definition}

\newtheorem{dfnt}[thm]{Definition}

\def\({\left(}

\def\){\right)}

\def \<{{\langle}}
\def \>{{\rangle}}

\numberwithin{equation}{section}

\title[Irr Repn of Tor Lie Alg]{Irreducible modules of toroidal Lie algebras arising from $\phi_{\e}$-coordinated modules of vertex algebras}

\author{Fulin Chen}
\address{School of Mathematical Sciences, Xiamen University,
Xiamen, 361005, China} \email{chenf@xmu.edu.cn }

\author{Huansheng Li}\address{School of Mathematical Sciences, Xiamen University,
Xiamen, 361005, China} \email{hslee@stu.xmu.edu.cn}

\author{Nina Yu}
\address{School of Mathematical Sciences, Xiamen University,
Xiamen, 361005, China} \email{ninayu@xmu.edu.cn}



\begin{document}
	
\begin{abstract}

In this paper, for every $\e\in \Z$, we introduce an extension of the 2-toroidal Lie algebra by certain derivations.
Based on the $\phi_\e$-coordinated modules theory for vertex algebras, we give an explicit realization of
 a class of irreducible highest weight modules for this extended toroidal Lie algebra.
When $\e=1$,  this affords a realization of certain irreducible modules for the toroidal extended affine Lie algebras first
constructed by Billig.

\end{abstract}

\maketitle
\section{Introduction}
Toroidal Lie algebras are  natural multi-variable generalizations of affine Lie
algebras. One starts with a finite-dimensional complex simple Lie algebra $\fg$ and
forms  the multi-loop algebra $\mathcal{R}\ot \fg$, where $\mathcal{R}= \C[t_0^{\pm 1},t_1^{\pm 1},\dots,t_N^{\pm 1}]$
is the ring of Laurent polynomials in the variables $t_0,t_1,\dots,t_N$.
The $(N+1)$-toroidal Lie algebra $\ft(\fg)$ is by definition  the universal central extension of $\mathcal{R}\ot \fg$
\cite{F,MRY,EM}.
When $N=0$, this yields the usual affine Lie algebra.
In  the literature, various extensions of $\ft(\fg)$
obtained by adding certain subalgebras of  $\Der (\mathcal R)$ (the Lie algebra of derivations over $\mathcal{R}$) have been studied extensively.
For examples, the subalgebra $\sum_{i=0}^N \C \rd_i(:= t_i\frac{\partial}{\partial t_i})$  of degree-zero derivations
\cite{E1,E2},
  the subalgebra $\mathcal{R} \rd_0\oplus \cdots\oplus \mathcal{R}\rd_{N-1}
  \oplus \C \rd_N$ \cite{EM,FM,BB,JM},
the subalgebra of divergence zero (or skew) derivations  \cite{B2,CLT1,ESB,CLT2},
and the algebra $\Der (\mathcal R)$ itself \cite{B1,EJ}.

In this paper we consider a sequence of new extensions
\[\wt\ft(\fg)^\e=\ft(\fg)\oplus \mathcal{D}^\e\quad (\e\in \Z)\] of
the $2$-toroidal Lie algebra $\ft(\fg)$, where
\begin{align}\label{intro:defde}
\mathcal{D}^\e=\bigg\{f_0t_0^{\e}\frac{\p}{\p t_0}+f_1\rd_1\mid
f_0,f_1\in \mathcal R,\ t_0^\e\frac{\p}{\p t_0}(f_0)+\rd_1(f_1)=0\bigg \}
\end{align}
 is a Lie subalgebra of $\Der (\mathcal R)$.
The Lie algebras $\mathcal{D}^\e$ are closely related to
the Lie algebras $\mathcal{B}(q)$, $q\in \C$ of Block type studied in \cite{CGZ,SXX1,SXX2,XZ}.
In fact, when $q=\e$, $\mathcal{B}(q)$ is a ``half" of $\mathcal{D}^\e$, which consists of those derivations  in \eqref{intro:defde}
satisfying the condition that $f_0,f_1$ are spanned by the
monomials $t_0^{m_0}t_1^{m_1}$ with $m_0,m_1\in \Z$ and $m_1\ge 0$.
In  particular, when $\e=1$, $\mathcal{D}^\e$ is the algebra of divergence zero  derivations
and  $\wt\ft(\fg)^\e$ is a nullity $2$ extended affine Lie algebra in the sense of \cite{AABGP},
which is commonly called toroidal extended affine Lie algebras \cite{B2}.

One of the most important classes of  modules for affine Kac-Moody algebras is the highest weight modules. Their  analogs for various extensions of toroidal Lie algebras have been extensively studied in \cite{B1,B2,BB,BBS,CLT1,CLT2,E1,E2,EJ,ESB,FM,JM}.
In this paper we will give an explicit realization of a class of irreducible highest weight
$\wt\ft(\fg)^\e$-modules. When $\e=1$, this construction have been previously obtained by Billig in \cite{B2}.

 The key  ingredient in our realization is the
theory of $\phi$-coordinated modules for vertex algebras,
which was first introduced in \cite{Li}
for the purpose of associating quantum affine algebras with (quantum) vertex algebra.
Here, $\phi$ is an associate
of the $1$-dimensional additive formal group (law) $F(z,w)=z+w$, which is by definition
 a formal series $\phi(w,z)\in \C((w))[[z]]$, satisfying the conditions
$$\phi(w,0)=w\quad \text{and}\quad  \phi(\phi(w,z_1),z_2)=\phi(w,z_1+z_2).$$
It was proved therein that for every $p(w)\in \C((w))$, $\phi(w,z)=e^{zp(w)\frac{d}{dw}}(w)$ is an associate
of $F(z,w)$ and every associate is of this form. When $p(w)=1$, the $\phi$-coordinated modules are just
the usual modules. When $p(w)=w^\e$ with $\e\in \Z$, set
\[\phi_\e(w,z)=e^{zw^\e\frac{d}{dw}}(w).\]
 The structure of $\phi_\e$-coordinated modules for vertex algebras was further studied in \cite{BLP}.
In particular, a Jacobi-like identity for $\phi_\e$-coordinated modules was established therein.

Now we give an outline of the structure of this paper. In Section 2, we  first introduce  the extended toroidal Lie algebra $\wt\ft(\fg)^\e$ and give the commutators in terms of certain generating functions on $\wt\ft(\fg)^\e$ (see Proposition \ref{prop:hgcomm}). We will work in a more general framework with $\fg$ an arbitrary Lie algebra equipped with an invariant symmetric bilinear form. Meanwhile, in the process of  adding $\mathcal{D}^\e$ to
$\ft(\fg)$, an abelian two-cocycle depending on any complex number $\mu$ will be involved.

In Section 3, we first recall some facts about affine-Virasoro algebras $\wh\fg\rtimes\mv$
 associated to $\fg$. Then we give the first main result of this paper in Theorem \ref{mainth}. Namely,  we construct  a class of irreducible $\wt\ft(\fg)^\e$-modules
 \begin{align}\label{intro:module}
  L_{\wh\fg\rtimes\mv}(\ell,24\mu\ell-2,U,\beta)\ot V_{\wh\fh}(\ell,e^{\alpha\bfk}\C[L])
 \end{align}
where  $\alpha,\beta,\ell\in \C$ with $\ell\ne 0$,  $U$ is an irreducible $\fg$-module,
$L_{\wh\fg\rtimes\mv}(\ell,24\mu\ell-2, U,\beta)$ is an irreducible highest weight module for the affine-Virasoro algebra
$\wh\fg\rtimes\mv$, and $V_{\wh\fh}(\ell,e^{\alpha\bfk}\C[L])$ is a highest weight module for the Heisenberg algebra
$\wh\fh$
associated to a $2$-dimensional abelian Lie algebra $\fh=\C\bfk\oplus \C\bfd$.
The second main result of this paper is a characterization of  these  irreducible $\wt\ft(\fg)^\e$-modules in terms of
the highest weight modules (see Theorem \ref{mainth2}).

Sections 4 and 5 are devoted to the proof of Theorem \ref{mainth}. In Section 4 we first recall some results on $\phi_\e$-coordinated modules.
 Let $V$ be a vertex algebra.  A fundamental property for a $V$-module $(W,Y_W)$ is that  for $u,v\in V$,
$Y_W(u_{-1}v,z)$ is just the normally ordered product $ _{\circ}^{\circ}Y_W(u,z)Y_W(v,z) _{\circ}^{\circ}$ (see \cite{LL} for example).
We prove in Proposition \ref{th7} that for a $\phi_\e$-coordinated $V$-module $(W,Y_W^\e)$,
\begin{align}\label{intro:u-1vformula}
Y_W^\e(u_{-1}v,z)= _{\circ}^{\circ}Y_W^\e(u,z)Y_W^\e(v,z) _{\circ}^{\circ}+\sum_{n\ge 0}c_n z^{(\e-1)(n+1)} Y_W^\e(u_n v,z),\end{align}
where $c_n$ are some explicitly determined complex numbers.
We also determine  the $\phi_\e$-coordinated modules for
vertex algebras arising from vertex Lie algebras introduced in \cite{DLM, Kac,P} (see  Proposition \ref{prop:vla2}).

In Section 5, we first introduce a subalgebra $\wh\ft(\fg)^\e$ of $\wt\ft(\fg)^\e$ such that
$\wt\ft(\fg)^\e=\wh\ft(\fg)^\e\oplus \C t_0^{\e-1}\rd_1$.
When $\epsilon=0$, $\wh\ft(\fg)^0$ is a vertex Lie algebra  and
therefore  there is a  vertex algebra
$V_{\wh\ft(\fg)^0}(\gamma_\ell)$ associated to $\wh\ft(\fg)^0$ for any nonzero complex number $\ell$ \cite{DLM}.
Based on a result of \cite{B1}, it was proved in \cite{CLiT} that there is a surjective vertex algebra homomorphism
\[\Theta:\quad V_{\wh\ft(\fg)^0}(\gamma_\ell)\longrightarrow V_{\wh\fg\rtimes \mv}(\ell,24\mu\ell-2)\otimes V_{(\fh,L)},\]   where
 $V_{\wh\fg\rtimes \mv}(\ell,24\mu\ell-2)$ is the affine-Virasoro vertex algebra and $V_{(\fh,L)}$ is the lattice vertex algebra associated to the pair $(\fh,L:=\Z\bfk)$ \cite{B1,LW}. Note that $V_{(\fh,L)}$ can be realized as a quotient of the affine vertex algebra
associated to an affine Lie algebra $\wh\fp$ \cite{LW}.
By applying Proposition \ref{prop:vla2}, we deduce in Section 5.1 the following three correspondences:
\vspace{3mm}
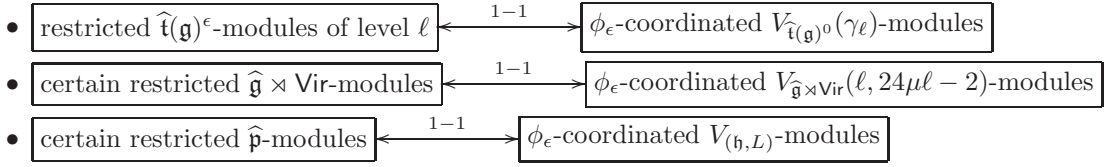
\begin{itemize}

\item $\xymatrix{*+[F]{\te{restricted $\wh\ft(\fg)^\e$-modules of level $\ell$}}\ar[rr]^-{1-1}&&*+[F]{\te{$\phi_\e$-coordinated $V_{\wh\ft(\fg)^0}(\gamma_\ell)$-modules}}\ar[ll]}$

\item $\xymatrix{*+[F]{\te{certain restricted $\wh\fg\rtimes \mv$-modules}}\ar[rr]^-{1-1}&&*+[F]{\te{$\phi_\e$-coordinated $V_{\wh\fg\rtimes \mv}(\ell,24\mu\ell-2)$-modules}}\ar[ll]}$

\item $\xymatrix{*+[F]{\te{certain restricted $\wh\fp$-modules}}\ar[rr]^-{1-1}&&*+[F]{\te{$\phi_\e$-coordinated $V_{(\fh,L)}$-modules}}\ar[ll]}$
\end{itemize}
\vspace{3mm}

The second correspondence  implies  that
there is an irreducible $\phi_\e$-coordinated $V_{\wh\fg\rtimes \mv}(\ell,24\mu\ell-2)$-module
structure on the  $\wh\fg\rtimes\mv$-module $L_{\wh\fg\rtimes\mv}(\ell,24\mu\ell-2, U,\beta)$ (see Proposition \ref{cocor}).
 Meanwhile, the third correspondence allows us to give an irreducible
 $\phi_\e$-coordinated $V_{(\fh,L)}$-module
structure on the  $\wh\fh$-module $V_{\wh\fh}(\ell,e^{\alpha\bfk}\C[L])$ (see Proposition \ref{vhaisvhlmod}).
Thus, the tensor product space  \eqref{intro:module} becomes an irreducible $\phi_\e$-coordinated $V_{\wh\fg\rtimes \mv}(\ell,24\mu\ell-2)\ot V_{(\fh,L)}$-module.
Via the homomorphism $\Theta$, this $\phi_\e$-coordinated $V_{\wh\fg\rtimes \mv}(\ell,24\mu\ell-2)\ot V_{(\fh,L)}$-module
carries a canonical (irreducible) $\phi_\e$-coordinate  $V_{\wh\ft(\fg)^0}(\gamma_\ell)$-module structure.
Then by applying the first correspondence above,
 the space  \eqref{intro:module} automatically admits
 an irreducible  $\wh\ft(\fg)^\e$-module structure.
This irreducible $\wh\ft(\fg)^\e$-module  can be extended to a $\wt\ft(\fg)^\e$-module  (see Theorem \ref{thm3}).
And, based on the formula \eqref{intro:u-1vformula}, we are able to compute the module action explicitly,
which is exactly what we state in Theorem \ref{mainth}.

In the rest of this paper, we let $\Z,\ \Z^{\times},\ \N,\ \C$ and $\C^{\times}$ be the sets of integers,
nonzero integers, nonnegative integers, complex numbers and nonzero complex numbers, respectively.
And, let $z,w,z_{0},z_{1},z_{2},\dots$ be mutually
commuting independent formal variables.
For a vector space $W$, $W[[z_1,z_2,\dots,z_r]]$ denotes the space of formal (possibly doubly
infinite) power series in $z_1,z_2,\dots,z_r$ with coefficients in $W$, and
$W((z_1,z_2,\dots,z_r))$  denotes the space of lower truncated Laurent power series in
$z_1,z_2,\dots,z_r$ with coefficients in $W$.

\section{Toroidal  Lie algebras } \label{sec2}
In this section we introduce the extended toroidal Lie algebras we concern about in this paper.

\subsection{Toroidal Lie algebras $\wt{\ft}(\fg)^\e$}
Let $\mathcal R=\C[ t_0^{\pm 1},  t_1^{\pm 1}]$ be
the Laurent polynomial ring in the variables $t_0$ and $t_1$. Denote by
\[\Omega_\mathcal R^1=\mathcal R \rd t_0\oplus \mathcal R \rd t_1=\mathcal R \rk_0\oplus \mathcal R\rk_1\]
the space of $1$-forms on $\mathcal R$, where $\rk_0:=t_0^{-1}\rd t_0$ and $\rk_1:=t_1^{-1}\rd t_1$.
Set
\[\rd({\mathcal R})=\{ \rd(f)\, |\, f\in \mathcal R\},\]
the space of exact $1$-forms, where
$$\rd(f):=\frac{\p f}{\p t_0}\rd t_0+\frac{\p f}{\p t_1}\rd t_1=
t_0\frac{\p f}{\p t_0}\rk_0+t_1\frac{\p f}{\p  t_1}\rk_1 \in \Omega_{\mathcal R}^1$$
is  the differential of $f$.
Form the quotient vector space
\begin{align*}
\mathcal K=\Omega_{\mathcal R}^1/ \rd({\mathcal R}).
\end{align*}
For  $m_0, m_1\in \Z$, set
\begin{equation*}
\rk_{m_0,m_1}=\begin{cases}\frac{1}{m_1} t_0^{m_0}t_1^{m_1}\rk_0,\ &\te{if}\ m_1\ne 0\\
-\frac{1}{m_0} t_0^{m_0}\rk_1,\ &\te{if}\ m_1=0,\ m_0\ne 0\\
0,\ &\te{if}\ m_0=m_1=0.\end{cases}
\end{equation*}
Then the set
\begin{align}\label{eq:basis1}
\{\rk_0, \rk_1\}\cup \{\rk_{m_0, m_1}\mid m_0,m_1\in \Z\ \te{with}\ (m_0,m_1)\ne (0,0)\}
\end{align}
 is a basis of $\mathcal K$,
noting that these  elements in $\Omega_\mathcal R^1$ should be understood as their images in $\mathcal{K}$.

Let $\fg$ be a Lie algebra equipped with a symmetric invariant bilinear form $\<\cdot,\cdot\>$.
Form a central extension of the double loop algebra $\mathcal R\ot \fg$:
\begin{align*}
\ft(\fg)=(\mathcal R\ot \fg)\oplus \mathcal K,
\end{align*}
 called the {\em $2$-toroidal Lie algebra}, where  $\mathcal K$ is central and
\begin{equation}\begin{split}\label{fre1}
[t_0^{m_0}t_1^{m_1}\ot u, t_0^{n_0}t_1^{n_1}\ot v]
=t_0^{m_0+n_0}t_1^{m_1+n_1}\ot [u,v]
+\< u,v\> \sum_{r=0, 1} m_r t_0^{m_0+n_0}t_1^{m_1+n_1}\rk_r
\end{split}
\end{equation}
for $u,v\in \fg$, $m_0,n_0, m_1,n_1\in \Z$.
If $\fg$ is  a  finite-dimensional simple Lie algebra and $\<\cdot,\cdot\>$  is nondegenerate,
$\ft(\fg)$ is known as a universal central extension of $\mathcal R\ot \fg$ (see \cite{MRY}).

We denote by
\[\Der (\mathcal R)=\mathcal{R}\frac{\p}{\p t_0}\oplus \mathcal{R}\frac{\p}{\p t_1}=
\mathcal{R}\rd_0\oplus \mathcal{R}\rd_1\] the
derivation Lie algebra on $\mathcal R$, where
$\rd_0:=t_0\frac{\p }{\p t_0}$ and $\rd_1:=t_1\frac{\p }{\p t_1}$.
By adding $\Der (\mathcal R)$ to $\ft(\fg)$, we obtain the \textit{full $2$-toroidal Lie algebra} (see \cite{B1})
\begin{align*}
\mathcal{T}(\fg)=\mathcal{T}(\fg)_\mu:=\ft(\fg)\oplus \Der (\mathcal R)
=(\mathcal{R}\ot \fg)\oplus \mathcal{K}\oplus \Der (\mathcal R),
\end{align*}
with Lie relations
\begin{alignat}{1}
\left[t_{0}^{m_{0}}t_{1}^{m_{1}}\rd_{i},t_{0}^{n_{0}}t_{1}^{n_{1}}\ot x\right] & =n_{i}(t_{0}^{m_{0}+n_{0}}t_{1}^{m_{1}+n_{1}}\ot x),\nonumber \\
\left[t_{0}^{m_{0}}t_{1}^{m_{1}}\rd_{i},t_{0}^{n_{0}}t_{1}^{n_{1}}\rk_{j}\right] & =n_{i}t_{0}^{m_{0}+n_{0}}t_{1}^{m_{1}+n_{1}}\rk_{j}+\delta_{i,j}\sum_{r=0,1}m_{r}t_{0}^{m_{0}+n_{0}}t_{1}^{m_{1}+n_{1}}\rk_{r},\nonumber \\
\left[pt_{0}^{m_{0}}t_{1}^{m_{1}}\rd_{i},t_{0}^{n_{0}}t_{1}^{n_{1}}\rd_{j}\right] & =n_{i}t_{0}^{m_{0}+n_{0}}t_{1}^{m_{1}+n_{1}}\rd_{j}-m_{j}t_{0}^{m_{0}+n_{0}}t_{1}^{m_{1}+n_{1}}\rd_{i}\label{eq:2.5}\\
 & -\mu m_{j}n_{i}\left(m_{0}t_{0}^{m_{0}+n_{0}}t_{1}^{m_{1}+n_{1}}\rk_{0}+m_{1}t_{0}^{m_{0}+n_{0}}t_{1}^{m_{1}+n_{1}}\rk_{1}\right)\nonumber
\end{alignat}
for $x\in \fg$, $m_0,n_0,m_1,n_1\in \Z$ and $i,j\in \{0,1\}$, where
$\mu$ is a fixed complex number  throughout this paper.
The algebra $\mathcal{T}(\fg)$ is $\Z$-graded with respect to the adjoint action of $-\rd_0$,
i.e., $$\mathcal{T}(\fg)=\oplus_{n\in \Z}\mathcal{T}(\fg)_{(n)},$$ where
\begin{align*}
\mathcal{T}(\fg)_{(n)}=\{x\in \mathcal{T}(\fg)\mid [\rd_0,x]=-nx\}.
\end{align*}

For a given integer $\e$, set
\begin{align*}
\mathcal{D}^\e=\bigg\{f_0t_0^{\e}\frac{\p}{\p t_0}+f_1\rd_1\mid
f_0,f_1\in \mathcal R,\ t_0^\e\frac{\p}{\p t_0}(f_0)+\rd_1(f_1)=0\bigg \}\subset \Der (\mathcal R).
\end{align*}
For $m_0,m_1\in \Z$, write
\begin{equation*}
\td_{m_0,m_1}^\e
=(m_0-\e+1)t_0^{m_0}t_1^{m_1}\rd_1-m_1t_0^{m_0}t_1^{m_1}\rd_0,
\end{equation*}
which is an element of $\mathcal{D}^\e$.
A simple fact is that the set
\begin{align}\label{eq:basis2}
\{t_0^{\e-1}\rd_0,t_0^{\e-1}\rd_1\}\cup
\{ \td^\e_{m_0,m_1}\mid (m_0,m_1)\in (\Z\times \Z)\backslash \{ (\e-1,0)\}\}
\end{align}
is a basis  of $\mathcal{D}^\e$.
Furthermore, $\mathcal{D}^\e$ is a Lie subalgebra of $\Der (\mathcal R)$ with the  relations
\begin{align*}
[t_0^{\e-1}\rd_0,\td^\e_{m_0,m_1}]
&=\(m_0-\e+1\)\td^\e_{m_0+\e-1,m_1},\quad [t_0^{\e-1}\rd_1,\td^\e_{m_0,m_1}]
=m_1\td^\e_{m_0+\e-1,m_1},\\
[\td^\e_{m_0,m_1},\td^\e_{n_0,n_1}]
&=\left((m_0-\e+1)n_1-m_1(n_0-\e+1)\right)\td^\e_{m_0+n_0,m_1+n_1},
\end{align*}
where $m_0, m_1, n_0, n_1\in \Z$.
Set
\begin{align*}
\wt\ft(\fg)^\e&=\ft(\fg) \oplus \mathcal{D}^{\e}\subset \mathcal{T}(\fg),
\end{align*}
which is the Lie algebra we study in this paper.
	
\begin{remt}{\em  When $\e=1$, the form $\<\cdot,\cdot\>$ on $\fg$ can be extended
to an invariant symmetric bilinear form $(\cdot,\cdot)$ on
the algebra $\wt\ft(\fg)^1$ such that the nontrivial values are as follows:
\[(t_0^{m_0}t_1^{m_1}\ot u, t_0^{-m_0}t_1^{-m_1}\ot v)=\<u,v\>,\quad
(\rk_{n_0,n_1},\rd_{-n_0,-n_1})=-1,\quad (\rk_i,\rd_i)=1\]
where $u,v\in \fg, m_0,m_1,n_0,n_1\in \Z$ with $(n_0,n_1)\ne (0,0)$ and $i=0,1$. Furthermore,
if $\fg$ is a finite-dimensional simple Lie algebra equipped with a Cartan subalgebra $H$, then the triple
\[\(\wt\ft(\fg)^1,\wt{H}:=H+\C\rk_0+\C\rk_1+\C\rd_0+\C\rd_1,(\cdot,\cdot)\)\] is an extended affine Lie algebra in the sense of
\cite{AABGP}, and is often referred as the nullity-2 toroidal extended affine Lie algebra \cite{B2,CLT1}.}
\end{remt}

\subsection{Relations of generating functions on $\wt\ft(\fg)^\e$}
	Consider the following subalgebra of $\wt\ft(\fg)^\e$:
\begin{align*}
\wt\fg_1=\(\C[t_1,t_1^{-1}]\ot \fg\)\oplus \C\rk_1\oplus \C\rd_1,
\end{align*}
which is isomorphic to the affine Kac-Moody algebra associated to $\fg$  \cite{Kac1}.
For any $m\in\mathbb{Z}$ and $u=t_1^n\otimes x + bk_1+cd_1\in\wt\fg_1$ with $x \in \fg$, $n\in \Z$ and $b, c\in \mathbb{C}$, we will often write
 \[t_0^mu:=t_0^mt_1^n\otimes x+bt_0^mk_1+ct_0^md_1\in \wt\ft(\fg)^\e.\]
In addition, for $m,n\in \Z$, set
\begin{align*}
\rd_{n,m}^\e
:=\tilde{\rd}_{n,m}^\e +\mu (1-\e)\( n+\frac{1}{2}(1-\e) \)m^2\rk_{n,m}.
\end{align*}
Let $\mathcal{B}$ be a vector space  with a designated basis
$\{ \rK_n,\, \rD_n\mid n\in \Z^{\times}\}.$  Set
\begin{align}\label{defag}\mathcal B_\fg=\widetilde\fg_1\oplus \mathcal{B}
=\(\C[t_1,t_1^{-1}]\ot \fg\)\oplus \C \rk_1\oplus \C \rd_1\oplus \sum_{n\in \Z^{\times}}\left(\C \rK_n\oplus \C \rD_n\right).
\end{align}

Define a linear map
\[\psi_\fg^\e:\mathcal{B}_\fg\rightarrow \wt\ft(\fg)^\e[[z,z^{-1}]],\quad a\mapsto a^\e(z)\] by letting
\begin{align*}
u^\e(z)&=\sum_{n\in \Z}(t_0^nu) z^{\e-n-1},\quad {
\rD_m^\e(z)=\sum_{n\in \Z} \rd^\e_{n,m} z^{2\e-n-2},\quad
\rK_m^\e(z)=\sum_{n\in\Z} \rk_{n,m} z^{-n},}
\end{align*}
where $u\in \wt\fg_1,\ m\in \Z^\times$.
Recall that $\rk_{m,0}=-\frac{1}{m}t_0^n\rk_1$ for $m\in \Z^\times$ and $\rd^\e_{n,0}=\tilde{\rd}^\e_{n,0}=(n-\e+1 ) t_0^n\rd_1$ for $n\in \Z$.
It follows from \eqref{eq:basis1} and \eqref{eq:basis2} that  the coefficients of $a^\e(z)$ for $a\in \mathcal B_\fg$ together with the elements  $\rk_0$ and $t_0^{\e-1}\rd_0$
linearly span the algebra $\wt\ft(\fg)^\e$.

For $a\in \C$ and $r\in \N$, set
\begin{align*}
a^{(r)}_\e=\prod_{s=0}^{r-1} (a+s(\e-1)).\end{align*}
The following  result will be used later:

\begin{lemt}\label{tech}
Let $p\in \N$, $a,b,\alpha,\beta\in\C$.
Then
\begin{align*}
\sum_{r=0}^p\binom{p}{r}\alpha^{p-r}{a_\e}^{(p-r)}(-\beta)^r {b_\e}^{(r)}
=\sum_{s=0}^{p}\binom{p}{s}(\alpha+\beta)^{p-s}a^{(p-s)}_\e \beta^s (-a-b-(p-1)(\e-1))^{(s)}_\e.
\end{align*}
\end{lemt}

\begin{proof}
Note that we have the following identity
\begin{align}\label{newton}
(a+b)^{(p)}_\e=\sum_{i=0}^{p} \binom{p}{i}a^{(i)}_\e b^{(p-i)}_\e,
\end{align}
which can be proved by induction on $p$.
Using \eqref{newton} and the facts
$$\binom{p-r}{t}\binom{p}{r}= \binom{p}{r+t}\binom{r+t}{r}\quad \(0\leqs r\leqs p,\ 0\leqs t\leqs p-r\), $$
and
$$a_\e^{(p-r)}=a_\e^{(p-s)}\(a+(p-s)(\e-1)\)_\e^{(s-r)} \quad  \(0\leqs r\leqs s\leqs p\),$$
we get
\begin{align*}
&\sum_{r=0}^p\binom{p}{r}\alpha^{p-r}a_\e^{(p-r)}(-\beta)^r b_\e^{(r)}\\
=\ &\sum_{r=0}^p\sum_{t=0}^{p-r}\binom{p-r}{t}\binom{p}{r}(\alpha+\beta)^{p-r-t}a_\e^{(p-r)}(-\beta)^{r+t} b_\e^{(r)}\\
=\ &\sum_{r=0}^p\sum_{t=0}^{p-r}\binom{p}{r+t}\binom{r+t}{r}(\alpha+\beta)^{p-r-t}a_\e^{(p-r)}(-\beta)^{r+t} b_\e^{(r)}\\
=\ &\sum_{s=0}^p\sum_{r=0}^{s}\binom{p}{s}\binom{s}{r}(\alpha+\beta)^{p-s}a_\e^{(p-r)}(-\beta)^{s} b_\e^{(r)}\\
=\ &\sum_{s=0}^p\sum_{r=0}^{s}\binom{p}{s}\binom{s}{r}(\alpha+\beta)^{p-s}
a_\e^{(p-s)}(a+(p-s)(\e-1))_\e^{(s-r)}(-\beta)^{s} b_\e^{(r)}\\
=\ &\sum_{s=0}^p\binom{p}{s}(\alpha+\beta)^{p-s}a_\e^{(p-s)}(a+b+(p-s)(\e-1))_\e^{(s)}(-\beta)^{s} \\
=\ &\sum_{s=0}^p\binom{p}{s}(\alpha+\beta)^{p-s}a_\e^{(p-s)}(-a-b-(p-1)(\e-1))_\e^{(s)}\beta^{s}.
\end{align*}
\end{proof}

The following proposition collects the Lie  relations among these generating functions.
\begin{prpt}\label{prop:hgcomm}
For $u,v\in \fg$, $a\in \mathcal{B}_{\fg}$, $m,n\in \Z$ and $k,l\in\Z^{\times}$, we have
\begin{align*}
&\ \te{(1)}\quad [(t_1^m\ot u)^\e(z),(t_1^n\ot v)^\e(w)] =\left(t_1^{m+n}\ot [u,v]\right)^\e(w)\zdwz\\
&\quad\quad \   +\<u,v\>m\(\wpw\rK^\e_{m+n}(w)\)\zdwz
+\<u,v\>(m+n)\rK^\e_{m+n}(w)\wpw\zdwz\\
&\quad\quad \  +\delta_{m+n,0}\<u,v\>\(m\rk_1^\e(w)\zdwz+\rk_0\wpw\zdwz \),\\
&\ \te{(2)}\quad [\rk_1^\e(z),(t_1^n\ot u)^\e(w)]=0=[\rK^\e_k(z),(t_1^n\ot u)^\e(w)],\\
&\ \te{(3)}\quad[\rD_k^\e(z),(t_1^n\ot u)^\e(w)]= k\(\wpw(t_1^{k+n}\ot u)^\e(w)\)\zdwz\\
&\quad\quad \ +(k+n) (t_1^{k+n}\ot u)^\e(w)\wpw\zdwz,\\
&\ \te{(4)}\quad [\rd_1^\e(z),(t_1^n\ot u)^\e(w)]=n(t_1^n\ot u)^\e(w)\zdwz,\\
&\ \te{(5)}\quad [\rk_1^\e(z),\rk_1^\e(w)]=[\rK^\e_k(z),\rK^\e_l(w)]=[\rk_1^\e(z),\rK^\e_l(w)]=0,\\
&\ \te{(6)}\quad [\rD_k^\e(z),\rK^\e_l(w)] =k\(\wpw \rK^\e_{k+l}(w)\)\zdwz+(k+l) \rK^\e_{k+l}(w)\wpw\zdwz \\
&\quad \quad \    +\delta_{k+l,0}\(k\rk_1^\e(z) \zdwz  +\rk_0\wpw\zdwz\),\\
&\ \te{(7)} \quad [\rD_k^\e(z),\rk_1^\e(w)] =k\wpw\(\(\wpw \rK^\e_k(w)\)\zdwz +\rK^\e_k(w)\wpw \zdwz \),\\
&\ \te{(8)}\quad [\rd^\e_1(z),\rK^\e_{n}(w)]=n\rK^\e_{n}(w)\zdwz,\\
&\ \te{(9)}\quad [\rd_1^\e(z),\rk_1^\e(w)]=\rk_0\zdwz,\quad\quad [\rd_1^\e(z),\rd_1^\e(w)]=0, \\
&\ \te{(10)}\ \  [\rD^\e_{k}(z),\rD^\e_{l}(w)] =k\(\wpw \rD^\e_{k+l}(w)\)\zdwz+(k+l) \rD^\e_{k+l}(w)\wpw\zdwz\\
&\quad \quad\  -k\delta_{k+l,0} \(\(\wpw \)^2 \rd_1^\e(w) \)\zdwz +\mu k^3\delta_{k+l,0}\(\(\wpw\)^2\rk_1^\e(w)  \)\zdwz\\
&\quad  \quad\  +\mu\sum_{r=0}^3{3\choose r}\(\(k\wpw\)^r \rK^\e_{k+l}(w)\)\((k+l)\wpw\)^{3-r}\zdwz,\\
&\ \te{(11)}\quad [\rd^\e_1(z),\rD^\e_{l}(w)]=l\rD^\e_{l}(w)\zdwz+\mu l^3 \rK^\e_l(w) \(\wpw\)^2 \zdwz,\\
&\ \te{(12)}\quad [t_0^{\e-1}\rd_0,a^\e(z)]=-z^{\e}\frac{d}{dz} a^\e(z),
\end{align*}
where $\rK_0^\e(z)$ and $\rD_0^\e(z)$ are understood as the zero formal series.
\end{prpt}

\begin{proof} We will only verify relation (10), which is the most complicated one. The verification of other relations are straightforward. Let $i,j,k,l\in \Z$.
Firstly, from \eqref{eq:2.5} we have
\begin{align*}
 & [\td_{i,k}^{\e},\td_{j,l}^{\e}]\\
=\ & [(i-\e+1)t_{0}^{i}t_{1}^{k}\rd_{1}-kt_{0}^{i}t_{1}^{k}\rd_{0},(j-\e+1)t_{0}^{j}t_{1}^{l}\rd_{1}-lt_{0}^{j}t_{1}^{l}\rd_{0}]\\
=\ & ((i-\e+1)l-k(j-\e+1))(1-\delta_{k+l,0})\td_{i+j,k+l}^{\e}-k\delta_{k+l,0}(i+j-\e+1)(i+j-2\e+2)t_{0}^{i+j}\rd_{1}\\
 & +\mu((i-\e+1)l-jk)(il-(j-\e+1)k)\ensuremath{(il-jk)\rk_{i+j,k+l}+\delta_{k+l,0}(kt_{0}^{i+j}\rk_{1}+i\delta_{i+j,0}\rk_{0})}\\
=\ & ((i-\e+1)l-k(j-\e+1))(1-\delta_{k+l,0})\ensuremath{\rd_{i+j,k+l}^{\e}-\mu(1-\e)(i+j+\frac{1}{2}(1-\e))(k+l)^{2}\rk_{i+j,k+l}}\\
 &  -k\delta_{k+l,0}(i+j-\e+1)(i+j-2\e+2)t_{0}^{i+j}\rd_{1}\\
 &  +\mu\delta_{k+l,0}((i+j-\e+1)^{2}k^{3}t_{0}^{i+j}\rk_{1}+(1-\e)^{2}k^{2}i\delta_{i+j,0}\rk_{0})\\
 &  +\mu((i-\e+1)l-jk)(il-(j-\e+1)k)(il-jk)(1-\delta_{k+l,0})\rk_{i+j,k+l}\\
=\ & ((i-\e+1)l-k(j-\e+1))(1-\delta_{k+l,0})\rd_{i+j,k+l}^{\e}-k\delta_{k+l,0}(i+j-\e+1)(i+j-2\e+2)t_{0}^{i+j}\rd_{1}\\
  & +\mu(1-\delta_{k+l,0})C_{k,l}^{i,j}\rk_{i+j,k+l}+\mu\delta_{k+l,0}((i+j-\e+1)^{2}k^{3}t_{0}^{i+j}\rk_{1}+(1-\e)^{2}k^{2}i\delta_{i+j,0}\rk_{0}),
\end{align*}
where $C_{k,l}^{i,j}$  stands for the following  number
\begin{align*}
&l^3(i-\e+1)\(i^2-(1-\e)(i+j+\frac{1}{2}(1-\e))\)\\
&-k^3(j-\e+1)\(j^2-(1-\e)(i+j+\frac{1}{2}(1-\e))\)\\
&-kl^2\(3(i-\e+1)i(j-\e+1)-(j+\frac{1}{2}(1-\e))(j+\e-1)\)\\
&+k^2l\(3(i-\e+1)i(j-\e+1)-(i+\frac{1}{2}(1-\e))(i+\e-1)\).
\end{align*}

Next, using  the fact
\[[\tilde\rd^\e_{i,k},\rk_{j,l}]=((i-\e+1)l-k(j+\e-1))\rk_{i+j,k+l},\] we obtain
\begin{align*}
&[\td_{i,k}^\e,\mu(1-\e)(j+\frac{1}{2}(1-\e))l^2\rk_{j,l}]
-[\td_{j,l}^\e,\mu(1-\e)(i+\frac{1}{2}(1-\e))k^2\rk_{i,k}]\\
=\ &\mu(1-\e)(1-\delta_{k+l,0})\bigg(((i-\e+1)l-k(j+\e-1))(j+\frac{1}{2}(1-\e))l^2\\
&-((j-\e+1)k-(i+\e-1)l)(i+\frac{1}{2}(1-\e))k^2\bigg)\rk_{i+j,k+l}\\
&+\mu(1-\e)\delta_{k+l,0}\bigg((j+\frac{1}{2}(1-\e))l^2((i-\e+1)\delta_{i+j,0}\rk_0+kt_0^{i+j}\rk_1)\\
&-(i+\frac{1}{2}(1-\e))k^2((j-\e+1)\delta_{i+j,0}\rk_0+lt_0^{i+j}\rk_1)\bigg)\\
=\ &\mu(1-\e)(1-\delta_{k+l,0})\bigg(((i-\e+1)l-k(j+\e-1))(j+\frac{1}{2}(1-\e))l^2\\
&-((j-\e+1)k-(i+\e-1)l)(i+\frac{1}{2}(1-\e))k^2\bigg)\rk_{i+j,k+l}\\
&+\mu(1-\e)\delta_{k+l,0}(k^3(i+j-\e+1)t_0^{i+j}\rk_1-k^2i(1-\e)\delta_{i+j,0}\rk_0).
\end{align*}

By  Lemma \ref{tech} with $p=3$, we have
\begin{align*}
&\sum_{r=0}^3\binom{3}{r}l^{3-r}(i-\e+1)_\e^{(3-r)}(-k)^r (j-\e+1)^{(r)}_\e\\
=\,&\sum_{s=0}^3\binom{3}{s}(k+l)^{3-s}(i-\e+1)_\e^{(3-s)}k^s (-i-j)^{(s)}_\e.
\end{align*}
Now by summarizing the above results, we find
\begin{align*}
 & [\rd_{i,k}^{\e},\rd_{j,l}^{\e}]\\
=\ & [\td_{i,k}^{\e}+\mu(1-\e)\left(i+\frac{1}{2}(1-\e)\right)k^{2}\rk_{i,k},\td_{j,l}^{\e}+\mu(1-\e)\left(j+\frac{1}{2}(1-\e)\right)l^{2}\rk_{j,l}]\\
=\ & (1-\delta_{k+l,0})((i-\e+1)l-k(j-\e+1))\rd_{i+j,k+l}^{\e}\\
 & +\delta_{k+l,0}(i+j-\e+1)(i+j-2\e+2)\left(-kt_{0}^{i+j}\rd_{1}+\mu k^{3}t_{0}^{i+j}\rk_{1}\right)\\
 & +\mu\left(1-\delta_{k+l,0}\right)\sum_{r=0}^{3}\binom{3}{r}l^{3-r}(i-\e+1)_{\e}^{(3-r)}(-k)^{r}(j-\e+1)_{\e}^{(r)}\rk_{i+j,k+l}\\
=\ & (1-\delta_{k+l,0})((i-\e+1)l-k(j-\e+1))\rd_{i+j,k+l}^{\e}\\
 & +\delta_{k+l,0}(i+j-\e+1)(i+j-2\e+2)\left(-kt_{0}^{i+j}\rd_{1}+\mu k^{3}t_{0}^{i+j}\rk_{1}\right)\\
 & +\mu(1-\delta_{k+l,0})\sum_{s=0}^{3}\binom{3}{s}(k+l)^{3-s}(i-\e+1)_{\e}^{(3-s)}k^{s}(-i-j)_{\e}^{(s)}\rk_{i+j,k+l}.
\end{align*}

According to the above commutator relation and the fact that
$$(i-\e+1)l-k(j-\e+1)=k(2\e-i-j-2)+(i-\e+1)(k+l),$$
we have
\begin{align*}
 & [\rD_{k}^{\e}(z),\rD_{l}^{\e}(w)]\\
=\ & \sum_{i,j\in\Z}[\rd_{i,k}^{\e},\rd_{j,l}^{\e}]z^{2\e-i-2}w^{2\e-j-2}\\
=\ & k(1-\delta_{k+l,0})\sum_{i,j\in\Z}\ensuremath{(2\e-i-j-2)\rd_{i+j,k+l}^{\e}w^{3\e-i-j-3}}z^{2\e-i-2}w^{i-\e+1}\\
 & +(k+l)\sum_{i,j\in\Z}\rd_{i+j,k+l}^{\e}w^{2\e-i-j-2}\ensuremath{(i-\e+1)z^{-i+2\e-2}w^{i}}\\
 & -k\delta_{k+l,0}\sum_{i,j\in\Z}\ensuremath{(\e-i-j-1)_{\e}^{(2)}t_{0}^{i+j}\rd_{1}w^{3\e-i-j-3}}z^{2\e-i-2}w^{i-\e+1}\\
 & +\mu(1-\delta_{k+l,0})\sum_{i,j\in\Z}\sum_{r=0}^{3}\binom{3}{r}\ensuremath{k^{r}(-i-j)_{\e}^{(r)}\rk_{i+j,k+l}w^{-i-j+r(\e-1)}}\\
 & \cdot\ensuremath{(k+l)^{3-r}(i-\e+1)_{\e}^{(3-r)}z^{2\e-i-2}w^{i+(2-r)(\e-1)}}\\
 & +\mu k^{3}\delta_{k+l,0}\sum_{i,j\in\Z}\ensuremath{(\e-i-j-1)_{\e}^{(2)}t_{0}^{i+j}\rk_{1}w^{3\e-i-j-3}}z^{2\e-i-2}w^{i-\e+1}\\
=\ & k\ensuremath{\wpw\rD_{k+l}^{\e}(w)}\zdwz+(k+l)\rD_{k+l}^{\e}(w)\wpw\zdwz\\
 & -k\delta_{k+l,0}\ensuremath{\ensuremath{\wpw}^{2}\rd_{1}^{\e}(w)}\zdwz\\
 & +\mu\sum_{r=0}^{3}{3 \choose r}\ensuremath{\ensuremath{k\wpw}^{r}\rK_{k+l}^{\e}(w)}\ensuremath{(k+l)\wpw}^{3-r}\zdwz\\
 & +\mu k^{3}\delta_{k+l,0}\ensuremath{\ensuremath{\wpw}^{2}\rk_{1}^{\e}(w)}\zdwz,
\end{align*}
proving (10).
\end{proof}

\section{Realization of irreducible highest weight $\wt{\ft}(\fg)^\e$-modules}\label{sec3}
In this section, as the main result of the paper, we present an explicit realization of certain irreducible highest weight
$\wt\ft(\fg)^\e$-modules.
\subsection{Highest weight modules for affine-Virasoro algebras}
Let $\fb$ be a  Lie algebra  equipped with an invariant
symmetric bilinear form $\<\cdot,\cdot\>$. We denote by
\begin{align*}
\wh\fb\rtimes \mv=\left(\mathrm{Der}\, \C[t,t^{-1}]\right)\ltimes  \left(\C[t,t^{-1}]\ot \fb\right)\oplus \C \rk\oplus \C\rk_{\mv},
\end{align*}
 the  {\em affine-Virasoro algebra} associated to  $\fb$, where $\rk$ and $\rk_{\mv}$ are central elements  and
\begin{equation}\begin{split}\label{eq:relaaffvir}
[L(m),L(n)]&=(m-n)L(m+n)+\frac{m^3-m}{12}\delta_{m+n,0}\rk_{\mv},\\
[u(m),v(n)]&=[u,v](m+n)+m\delta_{m+n,0}\<u,v\>\rk,\\
[L(m),u(n)]&=-nu(m+n)
\end{split}\end{equation}
for $u,v\in \fb$, $m,n\in \Z$ and
$$L(m):=-t^{m+1}\frac{d}{d t}, \quad u(m):=t^m\ot u\quad \text{ for }m\in \Z,\  u\in \fb.$$
Note that $\wh\fb\rtimes \mv$ contains the Virasoro algebra
\[\mv=\mathrm{Der}\, \C[t,t^{-1}]+\C\rk_{\mv}\]
and the affine algebras
\[\wh\fb=\(\C[t,t^{-1}]\ot \fb\)\oplus \C\rk,\quad \wt\fb=\wh\fb\oplus \C t\frac{d}{dt}
\]
as subalgebras.
For any $\epsilon\in \Z$, set
\begin{align}
\label{eq:affvirgene1} a^\e(z)&=\sum_{n\in\Z} \(a\ot t^n\) z^{\e-n-1}, \quad\forall a\in\fb,\\
\label{eq:affvirgene2} L^\e(z)&=\sum_{n\in \Z} L(n) z^{2\e-n-2}+\frac{\e^2-2\e}{24} z^{2(\e-1)}\rk_{\mv}.
\end{align}
When $\epsilon=0$, we also write $a(z)=a^0(z)$ and $L(z)=L^0(z)$.
 To emphasize  the dependence on $\fb$,  sometimes we will write $L_\fb(n)$, $L^\epsilon_\fb(z)$ and $L_\fb(z)$, instead of $L(n),L^\epsilon(z)$ and $L(z)$, respectively.

We fix a $\Z$-gradation on  $\wh\fb\rtimes \mv$ with respect to the adjoint action of $L(0)$:
\[\wh\fb\rtimes \mv=\oplus_{n\in \Z} (\wh\fb\rtimes \mv)_{(n)}.\]
This $\Z$-gradation  naturally  induces
a triangular decomposition  of $\wh\fb\rtimes \mv$:
\[\wh\fb\rtimes \mv=(\wh\fb\rtimes \mv)_{(+)}\oplus(\wh\fb\rtimes \mv)_{(0)}\oplus (\wh\fb\rtimes \mv)_{(-)},\]
where $(\wh\fb\rtimes\mv)_{(0)}=\fb\oplus\C L(0)\oplus\C\rk\oplus\C\rk_{\mv},$
and
\begin{gather*}
(\wh\fb\rtimes\mv)_{(\pm)}=\oplus_{n>0}(\wh\fb\rtimes\mv)_{(\mp n)}={\rm Span}\{u(\pm n),\ L(\pm n)\mid u\in\fb,\ n>0\}.
\end{gather*}
Let $U$ be a $\fb$-module and $\beta,\ell,c\in \C$.
View $U$ as a $(\wh\fb\rtimes \mv)_{(0)}$-module with $L(0)$, $\rk$ and $\rk_{\mv}$ act respectively as the scalars $\beta$, $\ell$ and $c$, and
extend it to a
$((\wh\fb\rtimes \mv)_{(+)} +(\wh\fb\rtimes \mv)_{(0)}) $-module
by  letting  $(\wh\fb\rtimes \mv)_{(+)}$ acts trivially.
Then we have the induced $\wh\fb\rtimes \mv$-module:
\begin{align}
V_{\wh\fb\rtimes \mv}(\ell,c,U,\beta)=\U(\wh\fb\rtimes \mv)
\ot_{\U\((\wh\fb\rtimes \mv)_{(+)} +(\wh\fb\rtimes \mv)_{(0)}  \)} U.
\end{align}
When the $\fb$-module $U$ is irreducible, $V_{\wh\fb\rtimes \mv}(\ell,c,U,\beta)$ has a unique irreducible quotient, called  $L_{\wh\fb\rtimes \mv}(\ell,c,U,\beta)$.
Note that the $\Z$-grading on $\wh\fb\rtimes \mv$ affords  naturally  $\N$-grading structures on $V_{\wh\fb\rtimes \mv}(\ell,c,U,\beta)$ and
$L_{\wh\fb\rtimes \mv}(\ell,c,U,\beta)$ with  $\deg U=0$.

Similarly, we have a $\Z$-grading $\wh\fb=\bigoplus \wh\fb_{(n)}$ and a triangular decompositions $\wh\fb=\wh\fb_{(+)}\oplus
\wh\fb_{(0)}\oplus \wh\fb_{(-)}$  of the affine Lie algebra $\wh\fb$, where  $\fb_{(n)}=t^{-n}\ot \fb+\delta_{n,0}\C\rk$ and $\fb_{(\pm)}=t^{\pm 1}\C[t^{\pm 1}]\ot \fb$.
Extend the $\fb$-module $U$ to  a $(\wh\fb_{(+)}\oplus \wh\fb_{(0)})$-module such that $\wh\fb_{(+)}$ acts trivially and $\rk$ acts as the scalar $\ell$.
Then we have the Verma type highest weight $\wh\fb$-module
\begin{align}\label{vblU}
V_{\wh\fb}(\ell,U)=\U(\wh\fb)\ot_{\U(\wh\fb_{(+)}+\wh\fb_{(0)})}U
\end{align}
and its irreducible quotient $L_{\wh\fb}(\ell,U)$ provided that $U$ is irreducible.
Note that both $V_{\wh\fb}(\ell,U)$ and $L_{\wh\fb}(\ell,U)$ are naturally $\N$-graded  with $U$ as the degree $0$ subspace.

We also have a triangular decomposition $\mv=\mv_{(+)}\oplus
\mv_{(0)}\oplus \mv_{(-)}$ of the Virasoro algebra $\mv$,
where $\mv_{(\pm)}=\oplus_{m>0}\C L(\pm m)$
and $\mv_{(0)}=\C L(0)+\C\rk_{\mv}$. View $\C$ as a $(\mv_{(+)}\oplus \mv_{(0)})$-module such that $\mv_{(+)}$ acts trivially, $L(0)$ acts as the scalar $\beta$ and $\rk_{\mv}$ acts
as the scalar $c$.
Then we have the Verma type highest weight module
\begin{align}
V_{\mv}(c,\beta)=\U(\mv)\ot_{\U(\mv_{(+)}+\mv_{(0)})}\C
\end{align}
and its irreducible quotient $L_{\mv}(c,\beta)$.

\begin{remt}
\emph{Assume that $\fb$ is  finite-dimensional simple, $\<\cdot,\cdot\>$ is the normalized form, and $U$ is finite dimensional.
By the Sugawara construction (see \cite{Kac1}), when $\ell$ is not the negative dual Coxeter number, the irreducible $\wh\fb\rtimes \mv$-module
$L_{\wh\fb\rtimes \mv}(\ell,c,U,\beta)$ can be realized as the tensor product of the $\wh\fb$-module
$L_{\wh\fb}(\ell,U)$ and  the $\mv$-module $L_{\mv}(c',\beta')$ for some suitable $c',\beta'\in \C$.	}
\end{remt}

\subsection{The main construction}

In this subsection we construct a class of irreducible
$\wt{\ft}(\fg)^\e$-modules.

Let $\fh=\C{\bf k}+\C{\bf d}$ be a vector space equipped with a designated basis $\{ {\bf k}, {\bf d}\}$ and
a symmetric bilinear form $\<\cdot,\cdot\>$ determined by
\begin{align}\label{formonh}
\<\bfk,\bfk\>=0=\<\bfd,\bfd\>,\quad \quad \<\bfk,\bfd\>=1.
\end{align}
Set $L=\Z\bfk\subset \fh$ and $\fh_{\bfk}=\C\ot_\Z L=\C\bfk$.
Denote by
\[\C[L]=\oplus_{\gamma\in L} \C e^\gamma\quad\text{and}\quad \C[\fh_{\bfk}]=\oplus_{h\in \fh_{\bfk}}\C e^{h}\]
the group algebras of $L$ and $\fh_\bfk$,
respectively.
We endow $\C[\fh_{\bfk}]$ with an
  $\fh$-module structure:  $h e^{h'}=\<h,h'\> e^{h'}$ for $h\in \fh, h'\in \fh_\bfk$.
  Note that for each $\alpha\in \C$, $e^{\alpha\bfk}\C[L]$ is an $\fh$-submodule of $\C[\fh_{\bfk}]$.
  Then for any $\ell\in \C^\times$, we have
a Verma type highest weight $\wh\fh$-module $V_{\wh\fh}(\ell,e^{\alpha\bfk}\C[L])$ (see \eqref{vblU}).
Note that we have the following vector spaces identification:
\[V_{\wh\fh}\left(\ell,e^{\alpha\bfk}\C[L]\right)=V_{\wh\fh}(\ell,0)\ot e^{\alpha\bfk}\C[L].\]

It is known (see \cite{LL}) that the highest weight $\wh\fh$-module $V_{\wh\fh}(\ell,e^{\alpha\bfk}\C[L])$ can be extended to an $\wh\fh\rtimes \mv$-module with
$\rk_{\mv}=2$ and
\begin{align}\label{Lfhz}L_\fh(z)=\frac{1}{\ell} \ _{\circ}^{\circ}\bfk(z)\bfd(z) _{\circ}^{\circ}.\end{align}
Let $W$ be a vector space and $u, v\in W$. For any two formal series $u(z)=\sum_{n\in\Z}u_nz^{-n-1}$ and $v(z)=\sum_{n\in\Z}v_nz^{-n-1}$ in $W[[z,z^{-1}]]$,
we denote by
\begin{align*}
 _{\circ}^{\circ}u(z_1)v(z_2) _{\circ}^{\circ}\ =u^+(z_1)v(z_2)+v(z_2)u^-(z_1)
\end{align*}	
the normally ordered product of $u(z_1)$ and $v(z_2)$, where
$$u^{+}(z)=\sum_{n\textless 0}u_nz^{-n-1}\quad \te{and}\quad u^-(z)=\sum_{n\geqs 0}u_nz^{-n-1}.$$

Form the direct sum
\[\ff=\fg\oplus \fh\]
of the Lie algebras $\fg$ and $\fh$.
Let $U$ be an irreducible $\fg$-module  and $\alpha,\beta,\ell\in \C$ with $\ell\ne 0$.
We define an $\wh\ff\rtimes \mv$-module structure on the tensor space
\[ L_{\wh\fg\rtimes\mv}(\ell,24\mu\ell-2,U,\beta)\ot V_{\wh\fh}(\ell,e^{\alpha\bfk}\C[L])\]
by letting $\rk=\ell$, $\rk_\mv=24\mu\ell$ and
\begin{align}\label{shortconvention1}
u(z)=u(z)\otimes 1,\quad h(z)=1\ot h(z),\quad L_\ff(z)=L_\fg(z)\otimes 1+1\otimes L_\fh(z),
\end{align}
for $u\in \fg$ and $h\in \fh$.

Let $\C[L]$ act on $V_{\wh\fh}(\ell,e^{\alpha\bfk}\C[L])$ by
\[e^\gamma (v\ot e^{\alpha\bfk+\gamma'})=
v\ot e^{\alpha\bfk+\gamma'+\gamma},\quad \forall \text{$\gamma,\gamma'\in L$, $v\in V_{\wh\fh}(\ell,0)$}.\]
For $\gamma\in L$,  introduce the following  operator in $\mathrm{End}\left(V_{\wh\fh}(\ell,e^{\alpha\bfk}\C[L])\right)[[z,z^{-1}]]$:
\begin{align}\label{Egammaz}
E^{\gamma}(z)=\te{exp}\(\frac{1}{\ell}\sum_{m=1}^{\infty} \frac{\gamma(-m)}{m}z^{m} \)
\te{exp}\(-\frac{1}{\ell}\sum_{m=1}^{\infty} \frac{\gamma(m)}{m}z^{-m} \)e^{\gamma}.
\end{align}
For convenience, we also set
\begin{align}\label{shortconvention2}
E^\gamma(z)=1\ot E^\gamma(z)\in \mathrm{End}\left( L_{\wh\fg\rtimes\mv}(\ell,24\mu\ell-2,U,\beta)\ot V_{\wh\fh}(\ell,e^{\alpha\bfk}\C[L])\right)[[z,z^{-1}]].\end{align}

The following theorem is the first main result of this paper. We will prove this theorem  in Section 5.
\begin{thmt}\label{mainth}
Let $\ell,\alpha,\beta\in\C$ with $\ell\neq 0$ and let $U$ be an irreducible $\fg$-module.
Then there is an irreducible $\wt\ft(\fg)^\e$-module structure on the $\wh\ff\rtimes \mv$-module
\[L_{\wh\fg\rtimes \mv}(\ell,24\mu\ell-2,U,\beta)\ot V_{\wh\fh}(\ell,e^{\alpha\bfk}\C[L]),\]
where $\rk_0=\ell$, $t_0^{\e-1}\rd_0=-L_\ff(\e-1)
+\delta_{\e,1}\mu\ell$, and
\begin{equation}\begin{split}\label{eq:mainaction}
&(t_1^m\ot u)^\e(z)=u^\e(z) E^{m\bfk}(z),\
\rk_1^\e(z)=\bfk^\e(z),\  \rd_1^\e(z)=\bfd^\e(z),\
\rK^\e_n(z)=\frac{\ell}{n} E^{n\bfk}(z) ,\\
&\rD^\e_n(z)=n _{\circ}^{\circ}L^\e_\ff(z)E^{n\bfk}(z)_{\circ}^{\circ}+\frac{1}{2}n\e (\e-1) z^{2\e-2} E^{n\bfk}(z)
-z^\e \frac{d}{dz} \ _{\circ}^{\circ} \bfd^\e(z) E^{n\bfk}(z) _{\circ}^{\circ}\\&\qquad\qquad +n^2 (\mu-\frac{1}{\ell})\(z^\e\frac{d}{dz}\bfk^\e(z)\)E^{n\bfk}(z)
\end{split}\end{equation}
for $u\in \fg,\ m\in \Z$ and $n\in \Z^\times$.
\end{thmt}

\begin{remt} When $\epsilon=1$, the $\wt\ft(\fg)^\e$-module $L_{\wh\fg\rtimes \mv}(\ell,24\mu\ell-2,U,\beta)\ot V_{\wh\fh}(\ell,e^{\alpha\bfk}\C[L])$
was first constructed by Billig in \cite{B2} (see also \cite{CLiT}).
And, when $\epsilon=0$, this  $\wt\ft(\fg)^\e$-module was previously constructed in \cite{CLiT}.
\end{remt}

\subsection{Realization of irreducible highest weight $\wt{\ft}(\fg)^\e$-modules}

In this subsection we give a characterization of the irreducible $\wt{\ft}(\fg)^\e$-modules
constructed in Theorem \ref{mainth}.

Note that  $\wt\ft(\fg)^\e=\oplus_{n\in \Z} \wt\ft(\fg)^\e_{(n)}$ is a $\Z$-graded subalgebra of $\mathcal{T}(\fg)$, where
$\wt\ft(\fg)^\e_{(n)}=\wt\ft(\fg)^\e\cap \mathcal{T}(\fg)_{(n)}$.
Explicitly, we have
\[
\wt\ft(\fg)_{(n)}^{\e}=\te{Span}\{t_{0}^{-n}u,\,\rd_{-n,m}^{\e},\,\rk_{-n,m}\mid m\in\Z,u\in\widetilde{\mathfrak{g}}_{1}\}\oplus\mathbb{C}\delta_{n,1-\epsilon}t_{0}^{\epsilon-1}\rd_{0},
\]
if $n\ne 0$, and
\[\wt\ft(\fg)_{(0)}^{\e}={\tilde{\fg}}_{1}+\sum_{n\in\Z^{\times}}(\C\rd_{0,n}^{\e}+\C\rk_{0,n})+\C\rk_{0}+\C\delta_{\e,1}\rd_{0}.\]
This $\Z$-grading induces  a natural triangular decomposition of $\wt\ft(\fg)^\e$:
\begin{align*}
\wt\ft(\fg)^\e=\wt\ft(\fg)^\e_{(+)}\oplus \wt\ft(\fg)^\e_{(0)}\oplus \wt\ft(\fg)^\e_{(-)},
\end{align*}
where  \[
\wt\ft(\fg)_{(\pm)}^{\e}=\oplus_{n>0}\wt\ft(\fg)_{(\mp n)}^{\e}.
\]

As usual, we say that a $\wt\ft(\fg)^\e$-module $W$ is a {\em highest weight module} if there exists a vector $w\in W$  such that
$\wt\ft(\fg)^\e_{(+)}w=0$ and $\U(\wt\ft(\fg)^\e)w=W$.
Let $T$ be an irreducible $\wt\ft(\fg)^\e_{(0)}$-module.  Then we have the Verma type highest weight $\wt\ft(\fg)^\e$-module
\[V_{\wt\ft(\fg)^\e}(T)=\U\(\wt\ft(\fg)^\e\)\ot_{\U\(\wt\ft(\fg)^\e_{(+)}+\wt\ft(\fg)^\e_{(0)}\)}T,\]
where $\wt\ft(\fg)^\e_{(+)}$ acts trivially on $T$.
The $\wt\ft(\fg)^\e$-module $V_{\wt\ft(\fg)^\e}(T)$ admits an $\N$-grading structure determined by $\deg T=0$.
Denote by  $L_{\wt\ft(\fg)^\e}(T)$ the  quotient of $V_{\wt\ft(\fg)^\e}(T)$ modulo by its maximal graded submodule.
Note that $L_{\wt\ft(\fg)^\e}(T)$ is an irreducible highest weight $\wt\ft(\fg)^\e$-module, and conversely any
irreducible highest weight $\wt\ft(\fg)^\e$-module is of this form.

Let $U$ be an irreducible $\fg$-module and $\ell,a,b\in \C$.
Associated to these data,  we will construct an irreducible $\wt\ft(\fg)^\e_{(0)}$-module structure on
the loop space $\C[t,t^{-1}]\ot U$.
Assume first that $\e\ne 1$. One can check that in this case
 \[\wt\ft(\fg)^\e_{(0)}=\C[t_1,t_1^{-1}]\ot \fg +\C\rk_1+\sum_{m\in\Z}\(\C\rd_{0,m}^\e+\C t_1^m\rk_0\),\]
where $\rk_1$ is central and
\begin{align*}
&[t_1^{m}\ot u,t_1^{n}\ot v]=t_1^{m+n}\ot [u,v]+m\delta_{m+n,0}\<u,v\>\rk_1,\ [t_1^m\rk_0,t_1^n\rk_0]=0,\ [t_1^m\rk_0,t_1^n\ot u]=0,\\
&[\rd^\e_{0,m},t_1^n\rk_0]=(1-\e)nt_1^{m+n}\rk_0-m^2\delta_{m+n,0}\rk_1,\
[\rd^\e_{0,m},t_1^{n}\ot u]=(1-\e)nt_1^{m+n}\ot u,\\
&[\rd^\e_{0,m},\rd^\e_{0,n}]=(1-\e)(n-m)\rd^\e_{0,m+n}+2\mu m^3 (\e-1)^2 \delta_{m+n,0} \rk_1
\end{align*}
for $m,n,k,l\in\Z$ and $u,v\in \fg$.
This implies that the assignment ($m\in \Z, u\in \fg$)
\begin{align}\label{homtg0}
\rk_1\mapsto 0,\quad \rd^\e_{0,m}\mapsto (1-\e)\,t^{m+1}\frac{d}{d t},\quad t_1^{m}\ot u\mapsto t^m\ot u,\quad t_1^m\rk_0\mapsto t^m\ot \rk_0
\end{align}
determines  a surjective homomorphism from $\wt\ft(\fg)^\e_{(0)}$ to the centerless affine-Virasoro algebra
\[\mathcal{W}_\fg:=(\mathrm{Der}\, \C[t,t^{-1}])\ltimes  \(\C[t,t^{-1}]\ot (\fg+\C\rk_0)\)\]
 associated to $\fg\oplus \C\rk_0$.
Extend $U$ to  a $(\fg\oplus \C\rk_0)$-module such that $\rk_0$ acts as a scalar $\ell$.
Following \cite{B-jet},  there is an irreducible $\mathcal{W}_\fg$-module structure on the loop space $\C[t,t^{-1}]\ot U$ with
\[(t^m\ot v)(t^n\ot u)=t^{m+n}\ot (vu),\quad \(t^{m+1}\frac{d}{dt}\) (t^n\ot u)= (n+a+b m) t^{m+n}\ot u\]
for $x\in \fg\oplus \C\rk_0$, $m,n\in \Z$ and $u\in U$.
Via the homomorphism \eqref{homtg0}, $\C[t,t^{-1}]\ot U$ becomes a $\wt\ft(\fg)^\e_{(0)}$-module,
which we denote by $T_{U,\ell,a,b}$.

When  $\e=1$ we have
\[\wt\ft(\fg)^1_{(0)}=\C[t_1,t_1^{-1}]\ot \fg +\sum_{m\in\Z}\(\C t_1^m\rd_0+\C t_1^m\rk_0\)+\C\rk_1+\C\rd_1,\]
which is isomorphic to the affine Kac-Moody algebra $\wt\ff$ associated to  $\ff$ with
\[t_1^m\ot u\mapsto u(m),\ t_1^m\rk_0\mapsto \bfk(m),\ t_1^m\rd_0\mapsto \bfd(m),\ \rk_1\mapsto \rk,\
\rd_1\mapsto t\frac{d}{dt}\]
for $m\in \Z, u\in \fg$.
Extend $U$ to an $\ff$-module such that $\bfd$ and $\bfk$ act as the scalars $b$ and $\ell$,
respectively. Following \cite{CP}, we have the loop module $\C[t,t^{-1}]\ot U$ for the affine Lie algebra $\wt\ff$ with the actions given by
\begin{align*}
x(m)(t^n\ot u)=t^{m+n}\ot (xu),\quad t\frac{d}{dt}(t^n\ot u)=(n+a)t^n \ot u,\quad \rk=0
\end{align*}
for $x\in \ff$, $m,n\in \Z$ and $u\in U$.
 Since $\wt\ft(\fg)_{(0)}^{1}$
is isomorphism to $\wt\ff$, the loop $\wt\ff$-module $\C[t,t^{-1}]\ot U$
becomes a $\wt\ft(\fg)_{(0)}^{1}$-module, which is also denoted by
$T_{U,\ell,a,b}$.

Now, as the second main result of this paper, we give a characterization of the irreducible $\wt\ft(\fg)^\e$-modules constructed in Theorem \ref{mainth}.

\begin{thm}\label{mainth2}
Let $\ell,\alpha,\beta\in\C$ with $\ell\neq 0$ and let $U$ be an irreducible $\fg$-module.
 Then the $\wt\ft(\fg)^\e$-module
$L_{\wh\fg\rtimes \mv}(\ell,24\mu\ell-2,U,\beta)\ot V_{\wh\fh}(\ell,e^{\alpha\bfk}\C[L])$ is isomorphic to
the irreducible highest weight module $L_{\wt\ft(\fg)^\e}(T_{\ell,U,\alpha,b})$, where $b=\frac{\beta+(\e^2-2\e)\mu\ell}{1-\e}+\frac{\e}{2}$
if $\e\ne 1$ and $b=\beta-\mu\ell$ if $\e=1$.
\end{thm}
\begin{proof}  Write \[W=L_{\wh\fg\rtimes \mv}(\ell,24\mu\ell-2,U,\beta)\ot V_{\wh\fh}(\ell,e^{\alpha\bfk}\C[L]).\]
It is clear that $W=\oplus_{k\in \Z}W_{(k)}$ is a $\Z$-graded $\wh\ff\rtimes \mv$-module such that $W_{(0)}=U\ot e^{\alpha\bfk}\C[L]$ and $W_{(k)}=0$ for $k<0$.
For $n\in \Z$, set
\[E^{n\bfk}(z)=\sum_{m\in \Z}E^{n\bfk}(m)z^{-m}.\]
 Then we have $E^{n\bfk}(m)W_{(k)}\subset W_{(m+k)}$ for $n,m,k\in \Z$.
This together with \eqref{eq:mainaction} implies that $\wt\ft(\fg)^\e_{(m)} W_{(k)}=W_{(m+k)}$ for $m,k\in \Z$.
That is, $W$ is an irreducible $\N$-graded  $\wt\ft(\fg)^\e$-module. Thus, $W_{(0)}$ is an irreducible
$\wt\ft(\fg)^\e_{(0)}$-module and we need to prove that
\[W_{(0)}(=U\ot e^{\alpha\bfk}\C[L])\quad \text{is isomorphic to}\quad T_{\ell,U,\alpha,b}(=\C[t,t^{-1}]\ot U)\]
as $\wt\ft(\fg)^\e_{(0)}$-modules, where $\ell,U,\alpha,b$ are as given in the theorem.

It is straightforward to check that in this case the $\wt\ft(\fg)^\e_{(0)}$-module actions  on
$T_{\ell,U,\alpha,b}$  can be written in the following expressions:
\begin{equation}\begin{split}\label{tg0action}
&(t_1^m\ot a) (t^n\ot u)=t^{m+n}\ot a u, \quad (t_1^m\rk_0)(t^n\ot u)=\ell t^{m+n}\ot u,\\
&\rd_1(t^n\ot u)=(n+\alpha)t^n\ot u,\quad \rk_1(t^n\ot u)=0,\\
& \rd^\e_{0,k} (t^n\ot u)=\(k(\beta+(\e^2-2\e)\mu\ell)+(1-\e)(\alpha+n+\frac{1}{2}k\e)\)t^{k+n}\ot u,\\
&\rd_0(t^n\ot u)=(\mu\ell-\beta)(t^n\ot u) \quad (\te{when } \e=1),
\end{split}\end{equation}
for $m,n\in \Z,\ k\in\Z^{\times},\ a\in \fg$ and $u\in U$.

Now we are ready to  show that $\wt\ft(\fg)^\e_{(0)}$-module actions on $W_{(0)}$ coincide with that of $T_{\ell,U,\alpha,b}$ given  in \eqref{tg0action}.
For $u\in U$ and $r\in \Z$, set
\[u(r):=u\ot e^{(\alpha+r)\bfk}.\] Then we have
\begin{align}\label{tge0act1}
\rk_1(u(r))&=\bfk(0)(u(r))
=u\ot \bfk(0)e^{(\alpha+r)\bfk}=\<\bfk,(\alpha+r)\bfk\>u\ot e^{(\alpha+r)\bfk}=0,\\
\label{tge0act2}\rd_1(u(r))&=\bfd(0)(u(r))
=\<\bfd,(\alpha+r)\bfk\>u(r)=(\alpha+r)u(r).
\end{align}
When $\e=1$, we also have
\begin{align}
\rd_{0}(u(r))= & (-L(0)-\sum_{i\textless1}\bfk_{i}\bfd_{-i}-\sum_{i\geq 1}\bfd_{-i}\bfk_{i}+\mu\ell)(u(r))\label{tge0act3}\\
= & (-L(0)+\mu\ell)(u(r))=(-\beta+\mu\ell)u(r).\nonumber
\end{align}

On the other hand, for $m\in\Z$, we have
\begin{align*}
E^{m\bfk}(0)\ensuremath{u(r)}=u\ot e^{m\bfk}e^{(\alpha+r)\bfk}=u(m+r).
\end{align*}
In view of this, we have
\begin{align}
(t_{1}^{n}\rk_{0})(u(r))= & n\rk_{0,n}(u(r))=\ell E^{n\bfk}(0)(u(r))=\ell(u(n+r)),\label{tge0act4}\\
(t_{1}^{m}\otimes a)(u(r))= & \sum_{i\in\Z}a(i)u\ot E^{m\bfk}(-i)E^{(\alpha+r)\bfk}\label{tge0act5}\\
=\, & a(0)u\ot E^{m\bfk}(0)e^{(\alpha+r)\bfk}=(au)(m+r),\nonumber
\end{align}
for $n\in\Z^{\times}$, $a\in\fg$ and $m\in\Z$.

For  the action of $\rd_{0,n}^\e$ $(n\in \Z^\times)$ on $u(r)$, note
firstly that
\begin{align*}
[L_\fh(i),E^{n\bfk}(-i)]e^{(\alpha+r)\bfk}=\left(\frac{n}{\ell}\sum_{j\in\Z}\bfk(j)E^{n\bfk}(-j)\right)e^{(\alpha+r)\bfk}=0\quad (i\in\Z),
\end{align*}
as $\bfk(i)$ commutes with $E^{n\bfk}(j)$ for $i,j\in\Z$, and that
$L_\fh(i)e^{(\alpha+r)\bfk}=0$ for $i\geq 0.$
Using these and \eqref{eq:affvirgene2} we have
\begin{align*}
 &\te{Res}_zz^{1-2\e}\ _{\circ}^{\circ}L^\e_\ff(z)E^{n\bfk}(z) _{\circ}^{\circ}u(r)\\
=\ & \te{Res}_z\( \sum_{i\leq 2\e-2}L_{\ff}(i)z^{-i-1}E^{n\bfk}(z)+E^{n\bfk}(z)\sum_{i> 2\e-2}L_{\ff}(i)z^{-i-1}+\frac{\e^2-2\e}{24}z^{-1}\rk_{\mv}E^{n\bfk}(z)\)u(r)\\
=\ &\( \sum_{i\leq 2\e-2}L_{\ff}(i)E^{n\bfk}(-i)+\sum_{i> 2\e-2}E^{n\bfk}(-i)L_{\ff}(i)+(\e^2-2\e)\mu\ell(E^{n\bfk})(0)\)u(r)\\
 =\ &\sum_{i\in \Z}\(L_\fg(i)u\ot E^{n\bfk}(-i)  e^{(\alpha+r)\bfk}+u\ot L_\fh(i) E^{n\bfk}(-i)  e^{(\alpha+r)\bfk}\)+(\e^2-2\e)\mu\ell(E^{n\bfk})(0)u(r)\\
  =\ &L_\fg(0)u\ot E^{n\bfk}(0)e^{(\alpha+r)\bfk}+(\e^2-2\e)\mu\ell\,u(n+r)\\
  =\ &(\beta+(\e^2-2\e)\mu\ell)\,u(n+r).
\end{align*}

 Secondly, we have
  \begin{align*}
 &\te{Res}_zz^{1-2\e}\( z^{2\e-2} E^{n\bfk}(z)\)u(r)
=(E^{n\bfk})(0)u(r)=u(n+r).
 \end{align*}
Thirdly, from the fact that $\bfk(i)$ commutes with $E^{n\bfk}(j)$ for $i,j\in\Z$,
we have
 \begin{align*}
&\te{Res}_zz^{1-2\e}\(z^\e\frac{d}{dz}\bfk^\e(z) \)E^{n\bfk}(z)u(r)\\
=\ &\te{Res}_z\sum_{i\in\Z}(\e-i-1)\bfk(i)z^{-i-1}E^{n\bfk}(z)u(r)\\
=\ & \sum_{i\in\Z}(\e-i-1)\bfk(i) E^{n\bfk}(-i)(u\ot e^{(\alpha+r)\bfk})=0.
 \end{align*}

Fourthly, using the relation $[\bfd(i),E^{n\bfk}(-i)]=n E^{n\bfk}(0)$ for $i\in\Z$, we have
 \begin{align*}
 &\te{Res}_zz^{1-2\e}\(z^\e\frac{d}{dz}\ _{\circ}^{\circ}\bfd^\e(z)E^{n\bfk}(z) _{\circ}^{\circ} \)u(r)\\
=\ &\te{Res}_z z^{1-\e}\frac{d}{dz}\ _{\circ}^{\circ}\bfd^\e(z)E^{n\bfk}(z)_{\circ}^{\circ}\ u(r)\\
=\ &\te{Res}_z z^{-\e}(\e-1)_{\circ}^{\circ}\bfd^\e(z)E^{n\bfk}(z)_{\circ}^{\circ}\ u(r)\\
=\ &(\e-1)\te{Res}_z\(\sum_{i\leq \e-1}\bfd(i)z^{-i-1}E^{n\bfk}(z)
+\sum_{i\textgreater \e-1}E^{n\bfk}(z)\bfd(i)z^{-i-1} \)u(r)\\
=\ &(\e-1) \(\sum_{i\leq \e-1}\bfd(i)E^{n\bfk}(-i)
+\sum_{i\textgreater \e-1}E^{n\bfk}(-i)\bfd(i)\)u(r)\\
=\ &\begin{cases}
(\e-1)\sum_{0\leq i\leq \e-1}\bfd(i)E^{n\bfk}(-i)u(r), &\te{if }\ \e\geq 1\\
(\e-1)\sum_{\e-1\textless i\leq 0}E^{n\bfk}(-i)\bfd(i)u(r), &\te{if }\ \e< 1
\end{cases}\\
=\ &(\e-1)(\alpha+r+n\e)\,u(n+r).
 \end{align*}

 Recall  that
\begin{align*}
\rd_{0,n}^{\e} & =\te{Res}_{z}z^{1-2\e}\rD_{n}^{\e}(z)\\
 & =\te{Res}_{z}z^{1-2\e}\Big(n\ _{\circ}^{\circ}L_{\ff}^{\e}(z)E^{n\bfk}(z)\ _{\circ}^{\circ}+\frac{1}{2}n\e(\e-1)z^{2\e-2}E^{n\bfk}(z)\\
 & \ \ \ -z^{\e}\frac{d}{dz}\ _{\circ}^{\circ}\bfd^{\e}(z)E^{n\bfk}(z) _{\circ}^{\circ}+n^{2}(\mu-\frac{1}{\ell})\ensuremath{z^{\e}\frac{d}{dz}\bfk^{\e}(z)}E^{n\bfk}(z)\Big).
\end{align*}
Then by combining the above identities together, we obtain
\begin{align}
\label{tge0act6}\rd_{0,n}^\e(u(r))
=(n(\beta+(\e^2-2\e)\mu\ell )+(1-\e)(\alpha+r+\frac{1}{2}n\e))(u(n+r)).
\end{align}

Finally, by comparing  \eqref{tge0act1}-\eqref{tge0act6} with  \eqref{tg0action}, we conclude that the map
\[T_{\ell,U,\alpha,b}=\C[t,t^{-1}]\otimes U\rightarrow  W_{(0)}=U\otimes e^{\alpha\bfk}\C[L],\quad t^r\ot u\mapsto u(r)\quad (r\in \Z, u\in U)\]
is a $\wt\ft(\fg)^\e_{(0)}$-module isomorphism.
This completes the proof.
\end{proof}

\begin{remt} Assume that $\fg$ is finite-dimensional and simple.
Note that $\rd_1$ acts semisimply on any irreducible highest weight $\wt\ft(\fg)^\e$-module $L_{\wt\ft(\fg)^\e}(T)$.
We say that $L_{\wt\ft(\fg)^\e}(T)$ is bounded if every graded subspace of it is a direct sum of finite dimensional $\rd_1$-eigenvalue spaces.
For example, if $U$ is finite-dimensional, then from Theorem \ref{mainth2} we find that every $L_{\wt\ft(\fg)^\e}(T_{\ell,U,\alpha,b})$  is bounded.
When $\e=0$, it was proved in \cite{CLiT} that every irreducible bounded highest weight $\wt\ft(\fg)^0$-module has such a form.
However, when $\e=1$, a large class of different irreducible bounded highest weight $\wt\ft(\fg)^1$-modules
 was constructed in \cite{CLT1}.
Thus, when $\e\ne 0,1$, it is interesting to classify the irreducible bounded highest weight $\wt\ft(\fg)^\e$-modules, which we { believe should}
 have the form $L_{\wt\ft(\fg)^\e}(T_{\ell,U,\alpha,b})$ with $U$ finite-dimensional.
\end{remt}

\section{Basics on $\phi_\e$-coordinated modules for vertex algebras}\label{sec4}

In this section, we collect some results on $\phi_\e$-coordinated modules for later use.
We denote a vertex algebra by $V=(V,Y, \mathbf{1})$ \cite{FHL,LL}, where ${\mathbf{1}}$ is the vacuum vector, and
$$Y(\cdot,z): V\rightarrow \mathrm{Hom}(V,V((z))),\ v\mapsto \sum_{n\in \Z} v_nz^{-n-1}$$ is the vertex operator.

\subsection{Basics on $\phi_\e$-coordinated modules}
As in Introduction, set
$$\phi_\e(z_2,z_0)=e^{z_0\(z_2^\e\frac{d}{dz_2}\)}z_2,$$
a particular associate of the one-dimensional additive formal
group $F\left(z,w\right)=z+w$ as defined in \cite{Li}. Now we recall
the notion of $\phi_\e$-coordinated module for a vertex algebra (see \cite{Li,BLP}).

\begin{dfnt}\label{defcoor}
{\em  Let $V$ be a vertex algebra. A {\em $\phi_{\e}$-coordinated $V$-module $\(W,Y_W^\e\)$} is a vector space $W$
equipped with a linear map
\[Y_{W}^\e(\cdot,z):V\rightarrow \mathrm{Hom}(W,W((z))),\quad v\mapsto Y_W^\e(v,z)\]
satisfying the conditions that $Y_W^\e({\bf 1},z)=1_W$ and that for $u,v\in V$,
there exists $k\in \N$ such that
\begin{align}\label{Lcommutator}
(z_1-z_2)^k Y_W^\e(u,z_1) Y_W^\e(v,z_2)&\in \mathrm{Hom}(W, W((z_1,z_2))),\\
(\phi_{\e}(z_2,z_0)-z_2)^k Y_W^\e(Y(u,z_0)v,z_2)&=\((z_1-z_2)^kY_W^\e(u,z_1)Y_W^\e(v,z_2)\)|_{z_1=\phi_{\e}(z_2,z_0)}.
\end{align}}
\end{dfnt}

We denote by $\mathcal{D}$ the canonical derivation  on $V$ defined by $v\mapsto v_{-2}\mathbf{1}$
for $v\in V$.
The following result was proved in \cite{Li}.

\begin{lemt}\label{lem:actofd} For a
$\phi_\e$-coordinated  $V$-module $\left(W,Y_{W}^{\e}\right)$,
 we have
\begin{align*}
Y_{W}^{\e}\left(\mathcal{D}v,z\right)=z\frac{d}{dz}Y_{W}^{\e}\left(v,z\right),\quad\forall v\in V.
\end{align*}
\end{lemt}

The following   Borcherds formula and Jacobi-type identity were obtained in \cite{BLP}:
\begin{prpt} Let $(W,Y_W^\e)$ be a $\phi_\e$-coordinated $V$-module. Then
for $u,v\in V$, we have
\begin{align}\label{Borcherds}
[Y_W^\e(u,z),Y_W^\e(v,w)]=\sum_{j\ge 0}\frac{1}{j!}Y_W^\e(u_jv,w)\(w^\e\frac{\p}{\p w}\)^jz^{\e-1}\delta\(\frac{w}{z}\),
\end{align}
and
\begin{equation}\begin{split}\label{eq:jacobi}
&(z_2z)^{-1}\delta\(\frac{z_1-z_2}{z_2z}\)Y_W^\e(u,z_1)Y_W^\e(v,z_2)-(z_2z)^{-1}\delta\(\frac{z_2-z_1}{-z_2z}\)Y_W^\e(v,z_2)Y_W^\e(u,z_1)\\
=\ &z_1^{-1}\delta\(\frac{z_2(1+z)}{z_1}\)Y_W^\e(Y(u,f_{\e}(z_2,z))v,z_2),
\end{split}\end{equation}
where
\begin{equation*}	f_{\e}(z_2,z)=\begin{cases}
		z_2^{1-\e}\cdot\frac{(1+z)^{1-\e}-1}{1-\e},\ &\text{for }\ \e\neq 1\\
		\te{log}(1+z),\ &\text{for }\ \e= 1.
\end{cases}\end{equation*}

\end{prpt}

As in the module case, we have the following result
whose proof is straightforward and is omitted (cf. \cite{FHL}).

\begin{lemt}\label{tensorphiemod} Let $V,V'$ be two vertex algebras,
$(W,Y_W^\e)$ a $\phi_\e$-coordinated $V$-module, and
$(W',Y_{W'}^\e)$  a $\phi_\e$-coordinated $V'$-module. Then
$(W\ot W', Y_W^\e\ot  Y_{W'}^\e)$
is a $\phi_\e$-coordinated module the tensor product vertex algebra $V\ot V'$.
\end{lemt}

In what follows, for $u,v\in V$, we consider the expressions of $Y_W^\e(u_0v,z)$ and $Y_W^\e(u_{-1}v,z)$ in a $\phi_\e$-coordinated $V$-module $(W,Y_W^\e)$.
\begin{lemt}\label{lem:u0v}Let $(W,Y_W^\e)$ be a $\phi_\e$-coordinated $V$-module. Then for $u,v\in V$,
\begin{align}
Y_W^\e(u_0v,w)=[\te{Res}_z z^{-\e}Y_W^\e(u,z), Y_W^\e(v,w)].
\end{align}
\end{lemt}
\begin{proof} By taking $\te{Res}_z z^{-\e}$ in the both sides  of \eqref{Borcherds}, we have
\begin{align*}
[\te{Res}_z z^{-\e}Y_W^\e(u,z), Y_W^\e(v,w)]=\te{Res}_z z^{-\e}Y_W^\e(u_0v,w)z^{\e-1}\delta\(\frac{w}{z}\)
=Y_W^\e(u_0v,w),
\end{align*}
noting that for $j>0$
\[\te{Res}_z z^{-\e}\(w^\e\frac{\p}{\p w}\)^jz^{\e-1}\delta\(\frac{w}{z}\)=\te{Res}_z z^{-\e}\(-z^\e\frac{\p}{\p z}\)^jz^{\e-1}\delta\(\frac{w}{z}\)=0.\]
\end{proof}

Set $h_0=1$ and
$$h_n=\sum_{\substack{r_1,r_2,\dots,r_n\ge0,\\
r_1+2r_2+\cdots+nr_n=n}}(-1)^{\sum\limits_{m=1}^{n}r_m}
\(\sum_{m=1}^nr_m\)!\prod_{m=1}^{n}\frac{1}{r_m!}\left(\frac{\e^{(m)}}{(m+1)!}\right)^{r_m},\ \forall n\ge 1,$$
where $\e^{(m)}=\prod_{s=0}^{m-1} (\e+s(\e-1))$.
We have
\begin{prpt}\label{th7}
Let $(W,Y_W^\e)$ be a $\phi_\e$-coordinated $V$-module. Then for $u,v\in V$,
\begin{align}\label{eq:u-1v}
&Y_W^\e(u_{-1}v,z)=_{\circ}^{\circ} Y_W^\e(u,z)Y_W^\e(v,z) _{\circ}^{\circ}-\sum_{n\geq 0}\sum_{i=0}^{n+1}\frac{1}{i!} \e^{(i)}
h_{n-i+1}z^{(n+1)(\e-1)}Y_W^\e(u_nv,z).
\end{align}
\end{prpt}
\begin{proof}
Note that we have
\begin{align*}
\phi_{\e}(z_2,z_0)-z_2
=\sum_{n\geq 1}\frac{1}{n!}z_0^n\(z_2^{\e}\frac{d}{d z_2}\)^nz_2
=z_0 z_2^{\e}h(z_2,z_0),
\end{align*} where
$$h(z_2,z_0)=\sum_{n\geq 0}\frac{1}{(n+1)!}\e^{(n)}z_2^{n(\e-1)}z_0^n\in \C((z_2))[[z_0]].$$
Since $h(z_2,0)=1\neq 0$, the inverse $h(z_2,z_0)^{-1}$ of $h(z_2,z_0)$ exists in $\C((z_2))[[z_0]]$. Explicitly, we have
\begin{align*}&h(z_2,z_0)^{-1}=\(1+\sum_{i\ge 1}\frac{\e^{(i)}}{(i+1)!}z_2^{i(\e-1)}z_0^i\)^{-1}
=\sum_{j\ge 0}\(-\frac{\e^{(i)}}{(i+1)!}z_2^{i(\e-1)}z_0^i\)^j
=\sum_{n\ge 0}h_nz_2^{n(\e-1)}z_0^n.
\end{align*}
Next we compute the residue $\te{Res}_{z}z^{-1}Y_W^\e(Y(u,f_{\e}(z_2,z))v,z_2)$.
By substituting $z$ by $\frac{\phi_{\e}(z_2,z_0)}{z_2}-1$ and using the facts that (see \cite{BLP})
$$ f_\e\(z_2,\frac{\phi_{\e}(z_2,z_0)}{z_2}-1\)=z_0\quad \te{and}\quad
\frac{\p}{\p z_0}\phi_{\e}(z_2,z_0)=\phi_{\e}(z_2,z_0)^{\e}, $$
we have
\begin{align*}
&\te{Res}_{z}z^{-1}Y_W^\e(Y(u,f_{\e}(z_2,z))v,z_2)\\
=\ & \te{Res}_{z_0}\(\frac{\phi_{\e}(z_2,z_0)}{z_2}-1\)^{-1}
\(\frac{\p}{\p z_0}\(\frac{\phi_{\e}(z_2,z_0)}{z_2}-1\)\)Y_W^\e(Y(u,z_0)v,z_2)\\
=\ & \te{Res}_{z_0}\(\phi_{\e}(z_2,z_0)-z_2\)^{-1}\phi_{\e}(z_2,z_0)^{\e} Y_W^\e(Y(u,z_0)v,z_2)\\
=\ & \te{Res}_{z_0}z_0^{-1}z_2^{-\e} h(z_2,z_0)^{-1}
e^{z_0\(z_2^{\e}\frac{\p}{\p z_2}\)}\(z_2^{\e}\)Y_W^\e(Y(u,z_0)v,z_2)\notag\\
=\ & \te{Res}_{z_0}\sum_{n\geqs -1}z_0^{-n-1}z_0^{-1}z_2^{-\e}
h(z_2,z_0)^{-1}e^{z_0\(z_2^{\e}\frac{\p}{\p z_2}\)}\(z_2^{\e}\)Y_W^\e(u_nv,z_2)\\
=\ & Y_W^\e(u_{-1}v,z_2)+ \sum_{n\geqs 0}\sum_{i=0}^{n+1}
 \frac{1}{i!} \e^{(i)} h_{n-i+1}z_2^{(n+1)(\e-1)}Y_W^\e(u_nv,z_2).
\end{align*}

On the other hand, by taking $\te{Res}_{z}\te{Res}_{z_1} z^{-1}$ in the Jacobi-type identity \eqref{eq:jacobi}, we find
\begin{align*}
&\text{Res}_{z} z^{-1}Y_W^\e(Y\(u,f_{\e}(z_2,z)\)v,z_2)\\
=\ & \te{Res}_{z_1}\big((z_1-z_2)^{-1}Y_W^\e(u,z_1)Y_W^\e(v,z_2)
-\(z_2-z_1\)^{-1}Y_W^\e(v,z_2)Y_W^\e(u,z_1)\big)\\
=\ &_{\circ}^{\circ}Y_W^\e(u,z_2)Y_W^\e(v,z_2) _{\circ}^{\circ}.
\end{align*}
This completes the proof of the proposition.
\end{proof}

\begin{remt}\label{rem:normalorder}
\emph{Here we give the first four terms in the expression \eqref{eq:u-1v} for later use:
\begin{align*}
Y_W^\e(u_{-1}v,z)=&_{\circ}^{\circ}Y_W^\e(u,z)Y_W^\e(v,z) _{\circ}^{\circ}-\frac{1}{2}\e z^{\e-1}Y_W^\e(u_0v,z)
-\frac{5\e -4}{12}\e z^{2(\e-1)}Y_W^\e(u_1v,z)\\
&-\frac{3\e^2-5\e+2}{8}\e z^{3(\e-1)}Y_W^\e(u_2v,z)+\cdots.\notag
\end{align*}}
\end{remt}
	
\subsection{On $\e$-deformation of vertex Lie algebras}
We start with a notion of vertex Lie algebra introduced in \cite{DLM} (see also \cite{Kac,P}).
A {\em vertex Lie algebra} is a quadruple $(\mathcal{L},\mathcal{A},\mathcal{Z},\eta)$, where $\mathcal{L}$
 is a Lie algebra,  $\mathcal A$ and $\mathcal Z$ are two vector spaces, and
\begin{eqnarray}
\eta:\ \(\C[t,t^{-1}]\ot \mathcal{A}\)\oplus {\mathcal{Z}}\rightarrow \mathcal{L},\quad
t^n\ot a+c\mapsto a(n)+c\quad (n\in \Z, a\in \mathcal{A}, c\in \mathcal{Z})
\end{eqnarray}
is a linear bijection, satisfying the condition that  $\eta(\mathcal{Z})$ is central in $\mathcal{L}$ and
 for $a,b\in {\mathcal{A}}$, there exist finitely many elements $a_{(i,j)}b\in \mathcal{A}$ and $a_{(j)}b\in \mathcal{Z}$ for $i,j\in \N$
such that
\begin{align}\label{az-bw}
[a(z),b(w)]=\sum_{i,j\in \N}\left(\left(\frac{\p}{\p w}\right)^i (a_{(i,j)}b)(w)+(a_{(j)}b)\right)\frac{1}{j!}\left(\frac{\partial}{\partial w}\right)^jz^{-1}\delta\left(\frac{w}{z}\right),
\end{align}
where $a(z)=\sum_{n\in \Z} a(n) z^{-n-1}$ for $a\in \mathcal{A}$.

Let $(\mathcal{L},\mathcal{A},\mathcal{Z},\eta)$ be a vertex Lie algebra.
Note that $\mathcal{L}^+=\eta(\C[t]\ot \mathcal{A}\oplus \mathcal{Z})$ is a subalgebra of $\mathcal{L}$.
For any  $\gamma\in \mathcal{Z}^*$,  denote by $\C_\gamma$ the one dimensional $\mathcal{L}^+$-module
such that $\eta(\C[t]\ot \mathcal{A})$ acts trivially and  $c\in \mathcal{Z}$ act as the scalars $\gamma(c)$.
Form the induced $\mathcal{L}$-module
\[V_{\mathcal{L}}(\gamma)=\mathcal{U}(\mathcal{L})\ot_{\mathcal{U}(\mathcal{L})}\C_\gamma.\]
Set ${\bm 1}=1\ot 1\in V_{\mathcal{L}}(\gamma)$ and identify $\mathcal{A}$ as a subspace of $V_{\mathcal{L}}(\gamma)$
with $a=a(-1){\bm 1}$ for $a\in \mathcal{A}$.
It was proved in \cite{DLM} that there is a unique vertex algebra structure on $V_{\mathcal{L}}(\gamma)$ with ${\bm 1}$ as the vacuum vector and
with $Y(a,z)=a(z)$ for $a\in \mathcal{A}$.

In what follows we define an $\e$-deformation Lie algebra $\mathcal{L}_\e$ of $\mathcal{L}$.
As a vector space, $\mathcal{L}_{\e}$ is a {linear copy} of $\mathcal{L}$ equipped with a linear bijection
\begin{align}\label{eq:etae}
\eta_\e:\ \(\C[t,t^{-1}]\ot \mathcal{A}\)\oplus {\mathcal{Z}}\rightarrow \mathcal{L}_\e,\quad
t^n\ot a+c\mapsto a^{(\e)}(n)+c^{(\e)}\quad (n\in \Z, a\in \mathcal{A}, c\in \mathcal{Z}).
\end{align}
We define a multiplication on $\mathcal{L}_\e$ such that $\eta_\e(\mathcal{Z})$ is central in $\mathcal{L}_\e$ and
\begin{align}\label{eq:relaLe}
&[a^{(\e)}(z),b^{(\e)}(w)]\\
=\,&\sum_{i,j\in \N}\left(\left(w^\e\frac{\p}{\p w}\right)^i (a_{(i,j)}b)^{(\e)}(w)+(a_{(j)}b)^{(\e)}\right)\frac{1}{j!}\left(w^\e\frac{\partial}{\partial w}\right)^jz^{\e-1}\delta\left(\frac{w}{z}\right),\notag
\end{align}
where $a^{(\e)}(z)=\sum_{n\in \Z} a^{(\e)}(n) z^{\e-n-1}$ for $a\in \mathcal{A}$.

We say that an $\mathcal{L}_\e$-module $W$ is {\em restricted} if  $a^{(\e)}(z)\in \mathrm{Hom}(W,W((z)))$ for any $a\in \mathcal{A}$,
 and is  of level $\gamma\in \mathcal{Z}^*$ if $c^{(\e)}$ act as the scalars $\gamma(c)$ for $c\in \mathcal{Z}$.

The main goal of this subsection is to prove the following two propositions.

\begin{prpt}\label{prop:vla1} Let $(\mathcal{L},\mathcal{A},\mathcal{Z},\eta)$ be a vertex Lie algebra. Then $\mathcal{L}_\e$ is a Lie algebra.
\end{prpt}

 \begin{prpt}\label{prop:vla2} Let $(\mathcal{L},\mathcal{A},\mathcal{Z},\eta)$ be a vertex Lie algebra and $\gamma\in \mathcal{Z}^*$.
Then the restricted $\mathcal{L}_\e$-modules $W$ of level $\gamma$ are
exactly
$\phi_\e$-coordinated $V_{\mathcal{L}}(\gamma)$-modules $(W,Y_W^\e)$  with $Y_W^\e(a,z)=a^{(\e)}(z)$ for $a\in \mathcal{A}$.
\end{prpt}

Now we prove the above two propositions. We will need to use a result on
$\phi_\e$-coordinated modules for the universal enveloping vertex algebra of a Lie conformal algebra from \cite{CLTW}.
Recall from \cite{Kac} that a {\em conformal algebra}  is a  $\C[\partial]$-module $\CC$ endowed with  $\C$-bilinear products $(a,b)\mapsto a_{j}b$ for $a,b\in \CC$ and $j\in\N$ such that
$a_jb=0$ for $j\gg 0$ and
\begin{align}\label{eq:conformalalg}
(\partial a)_{j}b=-ja_{j-1}b,\quad a_j(\partial b)=\partial(a_jb)+ja_{j-1}b.\end{align}
For a conformal algebra $\CC$, we define a
multiplication on its loop space $\C[t,t^{-1}]\ot\CC$ by
 \begin{eqnarray}\label{eq:relaCe}
 [a^{(\e)}(z),b^{(\e)}(w)]=\sum_{i\ge 0} \frac{1}{i!}(a_i b)^{(\e)}(w)\(w^\e\frac{\partial}{\partial w}\)^i z^{\e-1}\delta\(\frac{w}{z}\)
\end{eqnarray}
where $a,b\in\CC$ and $a^{(\e)}(z)=\sum_{n\in \Z} (t^n\ot a) z^{\e-n-1}$.
One can easily check that $\mathrm{Im}\, (1\ot \partial+t^\e\frac{d}{d t}\ot 1)$ is a two-sided ideal of $\C[t,t^{-1}]\ot \CC$ under this multiplication.
Form the quotient algebra
\begin{align*}
\wh{\CC}_\e:=\C[t,t^{-1}]\ot \CC/\mathrm{Im}\, (1\ot \partial+t^\e\frac{d}{d t}\ot 1),
\end{align*}
and denote by $\rho_\e:\C[t,t^{-1}]\ot \CC\rightarrow \wh{\CC}_\e$ the natural homomorphism.
For each $a\in \CC$, we will still denote
 the generating function $\sum_{n\in \Z} \rho_\e(t^n\ot a) z^{\e-n-1}$ in $\wh{\CC}_\e[[z,z^{-1}]]$ by $a^{(\e)}(z)$.

 A {\em Lie conformal algebra} is a conformal algebra $\CC$ such that for $a,b,c\in \CC$ and $m,n\in \N$,
\begin{eqnarray*}
a_{n}b=-\sum_{j=0}^{\infty}\frac{(-1)^{j+n}}{j!}\partial^j(b_{n+j}a)\quad \te{and}\quad
a_{m}(b_{n}c)=\sum_{j=0}^{m}\binom{m}{j}(a_{j}b)_{m+n-j}c+b_{n}(a_{m}c).
\end{eqnarray*}
The following results are from in \cite{P,Kac}.

\begin{lemt}\label{lem:lca} Let $\CC$ be a conformal algebra. Then $\CC$ is a Lie conformal algebra if and only if $\wh{\CC}_0$ is a Lie algebra.
Furthermore, in this case each $\wh{\CC}_\e$ is also a Lie algebra.
\end{lemt}

Let $\CC$ be a Lie conformal algebra. Set $\widehat{\CC}_0^+=\rho_0(\C[t]\ot \CC)$, a subalgebra of $\wh{\CC}_0$.
Let $\C$ be the trivial $\widehat{\CC}_0^+$-module and form the induced $\wh{\CC}_0$-module
\[V_{\CC}=\mathcal{U}(\wh{\CC}_0)\ot_{\mathcal{U}(\wh{\CC}_0^+)}\C.\]
Set $\bm{1}=1\ot 1$ and identify $\CC$ as a subspace of $V_{\CC}$ with $a\mapsto \rho_0(t^{-1}\ot a)\bm{1}$ for $a\in \CC$.
It is known that there is a unique vertex algebra structure on $V_{\CC}$ with ${\bm 1}$ as the vacuum vector and with
$Y(a,z)=a^0(z)$ for $a\in \CC$ \cite{P}.
We say that a  $\widehat{\CC}_{\e}$-module $W$ is {\em restricted} if  $a^{(\e)}(z)\in \mathrm{Hom}(W,W((z)))$ for any $a\in\CC$.
The following result was from \cite{CLTW}.
\begin{prpt}\label{prop:vacc} Let $\CC$ be a Lie conformal algebra. Then the restricted $\widehat{\CC}_{\e}$-modules $W$ are exactly
the $\phi_\e$-coordinated $V_{\CC}$-modules $(W,Y_W^\e)$ with $a^{(\e)}(z)=Y_W^\e(a,z)$ for $a\in \CC$.
\end{prpt}

\textbf{Proof of Propositions \ref{prop:vla1} and \ref{prop:vla2}.}
Associated to the vertex Lie algebra $\mathcal{L}$, we introduce  a vector space  as follows:
\begin{align*}
{\mathcal{CL}}=\(\C[\partial]\ot \mathcal{A}\)\oplus \mathcal{Z}.
\end{align*}
Endow $\mathcal{CL}$ with a $\C[\partial]$-module structure such that
$\partial(\partial^m\ot a)=\partial^{m+1}\ot a$ and $\partial(c)=0$
for $m\in \N, a\in \mathcal{A}$ and $c\in \mathcal{Z}$.
Furthermore, define $j$-products $(j\in \N)$ on $\mathcal{CL}$ such that
\[a_jb=a_{(j)}b+\sum_{i\ge 0} \partial^i(a_{(i,j)}b),\quad a_jc=c_ja=c_jc'=0\]
for $a,b\in \mathcal{A}$ and $c,c'\in \mathcal{Z}$, and then extend them to the whole space
$\mathcal{CL}$ via the rule \eqref{eq:conformalalg}.
This gives a conformal algebra structure on $\mathcal{CL}$.
Note that in $\wh{\mathcal{CL}}_\e$ we have $(\partial a)^{(\e)}(z)=z^\e\frac{\p}{\p z}a^{(\e)}(z)$ for $a\in \mathcal{A}$ and
$(\partial c)^{(\e)}(z)=z^\e\frac{\p}{\p z}c^{(\e)}(z)=0$ for $c\in \mathcal{Z}$.
This implies that  there is a linear isomorphism from $\wh{\mathcal{CL}}_\e$ to $\mathcal{L}_\e$ such that
\begin{align}
\rho_\e(t^n\ot a)\mapsto a^{(\e)}(n),\quad \rho_\e(t^{\e-1}\ot c)\mapsto c^{(\e)}\end{align}
for $n\in \Z, a\in \mathcal{A}$ and $c\in \mathcal{Z}$.
Furthermore, under this linear isomorphism, the multiplication \eqref{eq:relaCe} on $\wh{\mathcal{CL}}_\e$
coincides with the multiplication \eqref{eq:relaLe} on $\mathcal{L}_\e$.

Since $\mathcal{L}_0=\mathcal{L}$ is a Lie algebra,  $\wh{\mathcal{CL}}_0$ is  a Lie algebra as well.
This together with
 Lemma \ref{lem:lca} gives that $\mathcal{CL}$ is a Lie conformal algebra and $\wh{\mathcal{CL}}_\e$ is a Lie algebra.
Thus  $\mathcal{L}_\e$ is  a Lie algebra, which proves the Proposition \ref{prop:vla1}.
For the Proposition \ref{prop:vla2}, by identifying $\wh{\mathcal{CL}}_0$ with $\mathcal{L}$, one can easily check that
the vertex algebra $V_{\mathcal{L}}(\gamma)$ is isomorphic to the quotient vertex algebra of $V_{\mathcal{CL}}$
modulo the ideal generated by $c-\gamma(c)$ for $c\in \mathcal{Z}$.
According to Proposition \ref{prop:vacc}, $\phi_\e$-coordinated $V_{\mathcal{CL}}$-modules are exactly restricted
$\mathcal{L}_\e$-modules, and hence $\phi_\e$-coordinated $V_{\mathcal{L}}(\gamma)$-modules are exactly
restricted
$\mathcal{L}_\e$-modules of level $\gamma$.
This finishes the proof of proposition \ref{prop:vla2}.

\section{Proof of theorem \ref{mainth}}\label{sec6}

Here we finish the proof of Theorem \ref{mainth}.

\subsection{$\phi_\e$-coordinated modules for certain vertex algebras}

In this subsection we study the $\phi_\e$-coordinated modules for three types of vertex algebras.

First note  that the derived subalgebra of $\mathcal{D}^\e$ is
\[[\mathcal{D}^\e,\mathcal{D}^\e]=\te{Span}_\C\{\tilde{\rd}_{m,n}^\e\mid m,n\in \Z\}.\]
Set
\begin{align*}
\wh{\ft}(\fg)^\e=\ft(\fg)+[\mathcal{D}^\e,\mathcal{D}^\e]+\C t_0^{\e-1}\rd_1,
\end{align*}
which is a subalgebra of $\wt{\ft}(\fg)^\e$ such that $\wt{\ft}(\fg)^\e=\wh{\ft}(\fg)^\e\oplus \C t_0^{\e-1}\rd_0$.
Recall the space $\mathcal{B}_{\fg}$ defined in \eqref{defag}.
We define a linear isomorphism $\theta_\e:\(\C[t,t^{-1}]\ot \mathcal{B}_\fg\)\oplus \C\rk_0\rightarrow \wh{\ft}(\fg)^\e$ as follows:
\begin{align}
t^n\ot u\mapsto t_0^nu,\ t^n\ot \rD_m\mapsto \rd_{n+\e-1,m},\ t^n\ot \rK_m\mapsto
\rK_{n-\e+1,m},\ \rk_0\mapsto \rk_0
\end{align}
for $u\in \wt\fg_1, n\in \Z$ and  $m\in \Z^\times$.

In view of Proposition \ref{prop:hgcomm}, we see that the quadruple
\begin{align*}
(\mathcal{L}=\wh{\ft}(\fg)^0,\mathcal{A}=\mathcal{B}_{\fg},\mathcal{Z}=\C\rk_0,\eta=\theta_0)
\end{align*}
is a vertex Lie algebra.
Furthermore, the map $\eta_\e\circ \theta_\e^{-1}$ (see \eqref{eq:etae}) is a Lie isomorphism from $\wh{\ft}(\fg)^\e$ to $\wh{\ft}(\fg)^0_\e$.
Then for any complex number $\ell$, we have a vertex algebra $V_{\wh{\ft}(\fg)^0}(\gamma_\ell)$,
where $\gamma_\ell\in (\C\rk_0)^*$ defined by $\gamma_\ell(\rk_0)=\ell$.
We say that a $\wh{\ft}(\fg)^\e$-module $W$ is {\em restricted} if for any $a\in \mathcal{B}_{\fg}$, $a^\e(z)\in \mathrm{Hom}(W,W((z)))$,
and is {\em of level $\ell\in \C$} if $\rk_0$ acts as the scalar $\ell$ on $W$.
Then by Proposition \ref{prop:vla2} and the Lie isomorphism $\eta_\e\circ \theta_\e^{-1}:\wh{\ft}(\fg)^\e\rightarrow \wh{\ft}(\fg)^0_\e$,
we immediately have the following result.

\begin{prpt}\label{prop:phie1} For any $\ell\in \C$,
the restricted $\wh{\ft}(\fg)^{\e}$-modules of level $\ell$ are exactly the $\phi_\e$-coordinated
$V_{\wh{\ft}(\fg)^0}(\gamma_\ell)$-modules $(W,Y_W^\e)$ with
$Y_W^\e(a,z)=a^\e(z)$
for $a\in \mathcal{B}_{\fg}$.
\end{prpt}

Next, recall the affine-Virasoro algebra $\wh\fg\rtimes \mv$ defined in Section 3.
Form the vector spaces $\mathcal{A}=\fg\oplus \C \omega_{\fg}$ and $\mathcal{Z}=\C\rk\oplus \C \rk_{\mv}$. Define
a linear bijection
\[\eta:\(\C[t,t^{-1}]\ot \mathcal{A}\)\oplus \mathcal{Z}\rightarrow \wh\fg\rtimes \mv\] by letting
\begin{align*}
t^n\ot u\mapsto u(n),\quad t^n\ot \omega_{\fg}\mapsto L(n-1),\quad \rk\mapsto \rk,\quad \rk_{\mv}\mapsto \rk_{\mv}
\end{align*}
for $n\in \Z$, $u\in \fg$.
Note that in terms of generating functions
\begin{align*}
u(z)=\sum_{n\in \Z} u(n) z^{-n-1}\ (u\in \fg)\quad \text{and}\quad \omega_\fg(z)=\sum_{n\in \Z}\omega_\fg(n)z^{-n-1}
=\sum_{n\in \Z}L(n)z^{-n-2},
\end{align*}
the Lie relations in \eqref{eq:relaaffvir}  can be rewritten as follows:
\begin{equation}\begin{split}
[\omega_{\fg}(z),\omega_{\fg}(w)]&=\(\frac{\p}{\p w}\omega_\fg(w)\)z^{-1}\delta\(\frac{w}{z}\)
+2\omega_\fg(w)\frac{\p}{\p w}z^{-1}\delta\(\frac{w}{z}\)+\frac{1}{12}\rk_{\mv}\(\frac{\p}{\p w}\)^3z^{-1}\delta\(\frac{w}{z}\),\\
[u(z),v(w)]&=[u,v](w)z^{-1}\delta\(\frac{w}{z}\)+\<u,v\>\rk\frac{\p}{\p w}z^{-1}\delta\(\frac{w}{z}\),\\
[\omega_{\fg}(z),u(w)]&=\(\frac{\p}{\p w}u(w)\)z^{-1}\delta\(\frac{w}{z}\)+u(w)\frac{\p}{\p w}z^{-1}\delta\(\frac{w}{z}\)
\end{split}
\end{equation}
for $u,v\in \fg$.
This implies that the quadruple
\[(\mathcal{L}=\wh\fg\rtimes \mv, \mathcal{A}=\fg\oplus \C \omega_{\fg}, \mathcal{Z}=\C\rk\oplus \C \rk_{\mv}, \eta)\]
is a vertex Lie algebra.

For $\ell,c\in \C$, denote by $\gamma_{\ell,c}$ the linear functional on $\mathcal{Z}^*$ such that
$\gamma_{\ell,c}(\rk)=\ell$ and $\gamma_{\ell,c}(\rk_{\mv})=c$.
Then we have the (universal) affine-Virasoro vertex algebra
\[V_{\wh\fg\rtimes \mv}(\ell,c):=V_{\wh\fg\rtimes \mv}(\gamma_{\ell,c}),\] which is equipped with
a conformal vector $\omega_{\fg}$.
We say that a $\wh\fg\rtimes \mv$-module $W$ is {\em restricted} if for any $w\in W, a\in \fg$, one has $a(n)w=0=L(n)w$ for $n\gg 0$,
is {\em of level $\ell\in \C$} if $\rk$ acts as the scalar $\ell$, and is {\em of central charge $c$ } if $\rk_{\mv}$ acts as the
scalar $c$.
Recall the generating functions $a^\e(z)$ for $a\in \fg$ and $L^\e(z)$ defined in \eqref{eq:affvirgene1} and \eqref{eq:affvirgene2}, respectively.
Then we have:

\begin{prpt}\label{cocor} For $\ell,c\in \C$,
the restricted $\wh\fg\rtimes \mv$-modules $W$ of level $\ell$ and central charge $c$ are exactly the $\phi_\e$-coordinated
$V_{\wh\fg\rtimes \mv}(\ell,c)$-modules $(W,Y_W^\e)$ with
\begin{align}\label{actgvironw}
Y_W^\e(a,z)=a^\e(z), \forall a\in\fg; \quad Y_W^\e(\omega_\fg,z)=L_\fg^\e(z).
\end{align}

\end{prpt}
\begin{proof} By Proposition \ref{prop:vla1} we have an $\e$-deformation Lie algebra
\[(\wh\fg\rtimes \mv)_\e=\text{Span}_\C\{ u^{(\e)}(n), \omega_\fg^{(\e)}(n), \rk^{(\e)}, \rk_{\mv}^{(\e)}\mid u\in \fg, n\in \Z\}\]
of $\wh\fg\rtimes \mv$ with the Lie brackets given by
\begin{equation*}\begin{split}
[\omega_{\fg}^{(\e)}(z),\omega_{\fg}^{(\e)}(w)]&=\(w^\e\frac{\p}{\p w}\omega_\fg^{(\e)}(w)\)z^{\e-1}\delta\(\frac{w}{z}\)
+2\omega_\fg^{(\e)}(w)w^\e\frac{\p}{\p w}z^{\e-1}\delta\(\frac{w}{z}\)\\
&\ \ +\frac{1}{12}\rk_{\mv}^{(\e)}\(w^\e\frac{\p}{\p w}\)^3z^{\e-1}\delta\(\frac{w}{z}\),\\
[u^{(\e)}(z),v^{(\e)}(w)]&=[u,v]^{(\e)}(w)z^{\e-1}\delta\(\frac{w}{z}\)+\<u,v\>\rk^{(\e)}w^\e\frac{\p}{\p w}z^{\e-1}\delta\(\frac{w}{z}\),\\
[\omega_{\fg}^{(\e)}(z),u^{(\e)}(w)]&=\(w^\e\frac{\p}{\p w}u^{(\e)}(w)\)z^{\e-1}\delta\(\frac{w}{z}\)+u^{(\e)}(w)w^\e\frac{\p}{\p w}z^{\e-1}\delta\(\frac{w}{z}\)
\end{split}
\end{equation*}
for $u,v\in \fg$.

It is straightforward to check that  the linear map
\[
u^{(\e)}(n)\mapsto u(n),\ \omega_\fg^{(\e)}(n)\mapsto L(n+\e-1)+\delta_{n+\e-1,0}\frac{\e^2-2\e}{24}\rk_{\mv},\
\rk^{(\e)}\mapsto \rk,\ \rk_{\mv}^{(\e)}\mapsto \rk_{\mv}
\]
where $u\in \fg$ and $n\in \Z$, is a Lie algebra isomorphism from $(\wh\fg\rtimes \mv)_\e$ to $\wh\fg\rtimes \mv$.
 Then the assertion  follows immediately from Proposition \ref{prop:vla2}.
\end{proof}

Finally, recall from Section 3.2 that we have an $\wh\fh$-module
\[V_{\wh\fh}(\ell,e^{\alpha\bfk}\C[L])=V_{\wh\fh}(\ell,0)\ot e^{\alpha\bfk}\C[L].\]
When $\ell=1$ and $\alpha=0$, it is known that there is a vertex algebra structure on
\[V_{(\fh,L)}:=V_{\wh\fh}(1,\C[L])=V_{\wh\fh}(1,0)\ot \C[L]\]
 with $Y(h,z)=h(z)$ for $h\in \fh$ and
$Y(e^\gamma,z)=E^{\gamma}(z)$ (see \cite{B1} or \cite{LW} for example).
Note that $\omega_{\fh}=\bfk_{-1}\bfd$ is a conformal vector of $V_{(\fh,L)}$ of rank $2$, which is
also a conformal vector of the Heisenberg vertex subalgebra $V_{\wh\fh}(1,0)$.

View $\fh$ and $\C[L]$ as abelian Lie algebras and let $\fh$ act on $\C[L]$ by derivations with
\[h e^\gamma=\<h,\gamma\>e^\gamma\]
for $h\in H, \gamma\in L$. Form the semiproduct Lie algebra
\[\mathfrak{p}=\fh\ltimes \C[L],\]
on which we have
\[[h,h']=0=[e^\gamma,e^{\gamma'}],\quad [h,e^\gamma]=\<h,\gamma\>e^{\gamma}\]
for $h,h'\in \fh, \gamma,\gamma'\in L$.
Extend the form $\<\cdot,\cdot\>$ on $\fh$ (see \eqref{formonh}) to $\mathfrak{p}$ by letting $\<\mathfrak{p},\C[L]\>=0$, which is still
symmetric and invariant.
Recall that $\wh\fp$ is the affine Lie algebra
associated to the pair $(\fp,\<\cdot,\cdot\>)$,
which is obviously  a vertex Lie algebra with $\mathcal{A}=\fp$ and $\mathcal{Z}=\C\rk$.
Take $\gamma\in (\C\rk)^*$ such that $\gamma(\rk)=1$.
Then we have the universal affine vertex algebra
\[V_{\wh\fp}(1,0):=V_{\wh\fp}(\gamma).\]
The following relation between the vertex algebras $V_{(\fh,L)}$  and $V_{\wh\fp}(1,0)$ was proved in \cite{LW}.

\begin{lemt}\label{vhLandvp10} The vertex algebra $V_{(\fh,L)}$
is isomorphic to the quotient vertex algebra of $V_{\wh\fp}(1,0)$ modulo the ideal  generated by
the set
\begin{align}\label{genvhL}
\{e^0-{\bm 1},\ e^\gamma e^\lambda-e^{\gamma+\lambda},\ \mathcal{D}e^{\lambda}-\lambda(-1)e^{\lambda}\mid  \gamma,\lambda\in L\}.\end{align}
\end{lemt}

Recall that $u^\e(z)=\sum_{n\in \Z} u(n) z^{\epsilon-n-1}$ for $u\in \fp$.
In view of the above lemma, we have:

\begin{prpt}\label{VhLandhatp} The $\phi_\e$-coordinated $V_{(\fh,L)}$-modules $(W,Y_W^\e)$ are exactly the restricted $\wh\fp$-modules $W$ of level $1$  satisfying the conditions
that
\begin{align}\label{conditiononpevhL}
\(e^0\)^\e(z)=1,\quad \(e^{\gamma}\)^\e(z)\(e^{\lambda}\)^\e(z)=\(e^{\gamma+\lambda}\)^\e(z),\quad
z^\e\frac{d}{dz} \(e^\gamma\)^\e(z)=\gamma^\e(z) \(e^\gamma\)^\e(z)
\end{align}
for $\gamma,\lambda\in L$. And the correspondence is given by
$Y_W^\e(h,z)=h^\e(z)$ and $Y_W^\e(e^{\gamma},z)=\(e^{\gamma}\)^\e(z)$ for $h\in \fh$ and $\gamma\in L$.
\end{prpt}
\begin{proof}
It is routine to see  that the $\epsilon$-deformation  $\wh\fp_\e$ of $\wh\fp$ is isomorphic to $\wh\fp$ itself with
\[u^{(\e)}(n)\mapsto u(n)\quad \text{and}\quad \rk^{(\e)}\mapsto \rk\] for $u\in \fp, n\in \Z$.
Thus from Proposition \ref{prop:vla2} it follows that the $\phi_\e$-coordinated $V_{\wh\fp}(1,0)$-modules $(W,Y_W^\e)$ are exactly the
restricted $\wh\fp$-modules $W$ of level $1$  with
\begin{align}\label{acthatpmod}
Y_W^\e(u,z)=u^\e(z),\quad \forall\ u\in \fp. \end{align}

Lemma \ref{vhLandvp10}  implies that the $\phi_\e$-coordinated $V_{(\fh,L)}$-modules  are exactly
those $\phi_\e$-coordinated $V_{\wh\fp}(1,0)$-modules $(W,Y_W^\e)$ such that
  $Y_W^\epsilon(v,z)=0$ for all $v$ in the set $\eqref{genvhL}$.
  Thus, by  the correspondence \eqref{acthatpmod},
  the $\phi_\e$-coordinated $V_{(\fh,L)}$-modules  are exactly the restricted $\wh\fp$-modules $W$ of level $1$ such that
 the following relations hold:
\begin{align*}
Y_{W}^{\epsilon}(e^{0}-{\bm{1}},z) & =(e^{0})^{\epsilon}(z)-1=0,\\
Y_{W}^{\e}(e_{-1}^{\gamma}e^{\lambda}-e^{\lambda+\gamma},z) & = _{\circ}^{\circ}Y_{W}^{\e}(e^{\gamma},z)Y_{W}^{\e}(e^{\lambda},z) _{\circ}^{\circ}-Y_{W}^{\e}(e^{\gamma+\lambda},z)\\
 & =\ensuremath{\left(e^{\gamma}\right)}^{\epsilon}(z)\ensuremath{\left(e^{\lambda}\right)}^{\epsilon}(z)-\ensuremath{\left(e^{\gamma+\lambda}\right)}^{\epsilon}(z)=0,\\
Y_{W}^{\e}(\mathcal{D}e^{\gamma}-\gamma_{-1}e^{\gamma},z) & =z^{\e}\frac{d}{dz}Y_{W}^{\e}(e^{\gamma},z)- _{\circ}^{\circ}Y_{W}^{\epsilon}(\gamma,z)Y_{W}^{\epsilon}(e^{\gamma},z)_{\circ}^{\circ}\\
 & =z^{\epsilon}\frac{d}{dz}\left(\ensuremath{e^{\gamma}}\right)^{\epsilon}(z)-\gamma^{\epsilon}(z)\ensuremath{\left(e^{\gamma}\right)}^{\epsilon}(z)=0,
\end{align*}
for $\gamma,\lambda\in L$.
Here the second relation follows from \eqref{eq:u-1v} and the fact that  $e^\gamma$ commutes with $e^\lambda$,
and the third relation follows from Lemma \ref{lem:actofd},  \eqref{eq:u-1v} and the fact that  $\gamma$ commutes with $e^\gamma$.
 This completes the proof.
 \end{proof}

Recall the operator $E^{\gamma}(z)$ defined in \eqref{Egammaz}. The following result shows that
there is a canonical $\phi_\e$-coordinated $V_{(\fh,L)}$-module structure on the $\wh\fh$-module $V_{\wh\fh}(\ell,e^{\alpha\bfk}\C[L])$.
Similar results for lattice vertex algebras have been obtained in \cite{JKLT}.

\begin{prpt}\label{vhaisvhlmod} For every $\alpha\in \C$ and $\ell\in \C^\times$, the $\wh\fh$-module $W:=V_{\wh\fh}(\ell,e^{\alpha\bfk}\C[L])$
admits an irreducible $\phi_\e$-coordinated $V_{(\fh,L)}$-module structure $Y^\e_W$, which is uniquely determined by
\begin{align}\label{vhlactonvhl}
Y^\e_W(\bfk,z)=\frac{1}{\ell}\bfk^\epsilon(z),\quad Y^\e_W(\bfd,z)
=\bfd^\epsilon(z),\quad\mathrm{and}\quad Y^\e_W(e^\gamma,z)
=E^{\gamma}(z),\ \forall\gamma\in L.\end{align}
Furthermore, the action of the conformal vector $\omega_\fh$ is (see \eqref{eq:affvirgene2} and \eqref{Lfhz})
\begin{align}\label{YeWomegafh}
Y^\e_W(\omega_\fh,z)=L_\fh^\e(z).
\end{align}
\end{prpt}
\begin{proof}  We first claim that the assignment
\begin{align}\label{whfpact}
\bfk(z)\mapsto \frac{1}{\ell}\bfk(z),\quad \bfd(z)\mapsto \bfd(z),\quad \rk\mapsto 1\quad \te{and}\quad e^\gamma(z)\mapsto z^{-\e}E^{\gamma}(z)\  ( \gamma\in L)
\end{align}
gives an irreducible  restricted $\wh\fp$-module structure on the $\wh\fh$-module $V_{\wh\fh}(\ell,e^{\alpha\bfk}\C[L])$
and on which the three conditions stated in \eqref{conditiononpevhL} hold.

Indeed, the commutators among those operators  in \eqref{whfpact} are as follows (see \cite[Proposition 6.5.2]{LL} for example)
\begin{align*}
[\bfk(z),\bfd(w)]=\frac{\partial}{\partial w}z^{-1}\delta\(\frac{w}{z}\),\quad [\bfk(z),\bfk(w)]=0=[\bfd(z),\bfd(w)],\\
[\bfk(z),E^{\gamma}(w)]=0,\quad [\bfd(z),w^{-\e}E^{\gamma}(w)]=\<\bfd,\gamma\>w^{-\e}E^{\gamma}(w)z^{-1}\delta\(\frac{w}{z}\).
\end{align*}
This implies that $V_{\wh\fh}(\ell,e^{\alpha\bfk}\C[L])$ is  a restricted  $\wh\fp$-module of level $1$ (with the action \eqref{whfpact}).
Meanwhile, it is known that (see \cite[(6.5.64)]{LL} for example)
\[ E^{\lambda}(z)E^{\gamma}(z)=E^{\lambda+\gamma}(z),\quad
\frac{d}{dz} E^{\lambda}(z)=\lambda(z)E^{\lambda}(z)
\]
for $\gamma,\lambda\in L$.
From this and  the fact that  $\(e^\gamma\)^\e(z)=E^{\gamma}(z)$, one can conclude  that
 this  $\wh\fp$-module satisfies the conditions in \eqref{conditiononpevhL}, as desired.

In view of Proposition \ref{VhLandhatp}, the above claim implies that the assignment \eqref{whfpact} determines
a $\phi_\e$-coordinated $V_{(\fh,L)}$-module structure on $V_{\wh\fh}(\ell,e^{\alpha\bfk}\C[L])$.
The irreducibility of this $V_{(\fh,L)}$-module is obvious, and it remains to deduce the action $\omega_\fh$ on $W$.

Note that in $V_{\fh,L}$ we have $\bfk_n\bfd=0$ for $n\ge 0$ unless $n=1$, in which case $\bfk_1\bfd={\bm 1}$.
Then from Remark \ref{rem:normalorder} we have
\begin{align*}
Y_{W}^{\e}(\omega_{\fh},z) & =Y^{\e}(\bfk_{-1}\bfd,z)\\
 & =_{\circ}^{\circ}Y_{W}^{\e}(\bfk,z)Y_{W}^{\e}(\bfd,z)_{\circ}^{\circ}-\frac{5\e-4}{12}\e z^{2(\e-1)}\\
 & =\frac{1}{\ell}\ _{\circ}^{\circ}\bfk^{\e}(z)\bfd^{\e}(z)_{\circ}^{\circ}-\frac{5\e-4}{12}\e z^{2(\e-1)}.
\end{align*}
Recall from \eqref{Lfhz} that $L_\fh(z)=\frac{1}{\ell}\ _{\circ}^{\circ}\bfk^\e(z)\bfd^\e(z)_{\circ}^{\circ}$, and from \eqref{eq:affvirgene2} that
\[L_\fh^\e(z)=z^{2\e}L_\fh(z)+\frac{\e^2-2\e}{12} z^{2(\e-1)}=\frac{1}{\ell}z^{2\e}\ _{\circ}^{\circ}\bfk(z)\bfd(z)_{\circ}^{\circ}+\frac{\e^2-2\e}{12} z^{2(\e-1)}.\]

By the definition of normally ordered products, in the $\wh\fh$-module
$V_{\wh\fh}(\ell,e^{\alpha\bfk}\C[L])$ we have
\begin{align*}
\frac{1}{\ell}\ _{\circ}^{\circ}\bfk^{\e}(z)\bfd^{\e}(z)_{\circ}^{\circ} & =\frac{1}{\ell}\ensuremath{\sum_{n<\e}\bfk(n)z^{\e-n-1}\bfd^{\e}(z)+\bfd^{\e}(z)\sum_{n\ge\e}\bfk(n)z^{\e-n-1}}\\
 & =z^{2\e}\frac{1}{\ell}\ensuremath{\sum_{n<\e}\bfk(n)z^{-n-1}\bfd(z)+\bfd(z)\sum_{n\ge\e}\bfk(n)z^{-n-1}}.
\end{align*}
This implies that
\[
\begin{split} & \frac{1}{\ell}\ _{\circ}^{\circ}\bfk^{\e}(z)\bfd^{\e}(z)_{\circ}^{\circ}-z^{2\e}\frac{1}{\ell}\ _{\circ}^{\circ}\bfk(z)\bfd(z)_{\circ}^{\circ}\\
=\, & \begin{cases}
z^{2\e-2}\cdot\sum_{0\le n<\e}\frac{1}{\ell}[\bfk(n),\bfd(-n)],\quad & \te{if}\ \ \e>0\\
0,\quad & \te{if}\ \ \e=0\\
z^{2\e-2}\cdot\sum_{\e\le n<0}\frac{1}{\ell}[\bfd(-n),\bfk(n)],\quad & \te{if}\ \ \e<0
\end{cases}\\
=\, & z^{2\e-2}\frac{1}{2}\e(\e-1).
\end{split}
\]

Summarizing the above relations we obtain
\begin{align*}
Y_{W}^{\e}(\omega_{\fh},z) & =\frac{1}{\ell}\ _{\circ}^{\circ}\bfk^{\e}(z)\bfd^{\e}(z)_{\circ}^{\circ}-\frac{5\e-4}{12}\e z^{2(\e-1)}\\
 & =z^{2\e}\frac{1}{\ell}\ _{\circ}^{\circ}\bfk(z)\bfd(z)_{\circ}^{\circ}+z^{2\e-2}\frac{1}{2}\e(\e-1)-\frac{5\e-4}{12}\e z^{2(\e-1)}\\
 & =\frac{1}{\ell}z^{2\e}\ _{\circ}^{\circ}\bfk(z)\bfd(z)_{\circ}^{\circ}+\frac{\e^{2}-2\e}{12}z^{2(\e-1)}=L_{\fh}^{\e}(z).
\end{align*}
This finishes the proof.
\end{proof}

\subsection{Proof of theorem \ref{mainth}}\label{sec6}
This section is devoted to a proof  of Theorem \ref{mainth}.

By taking the tensor product of the latter two vertex algebras considered in last subsection,	
for every nonzero complex number $\ell$, we have a conformal vertex algebra
\[V_{\wh\fg\rtimes \mv}(\ell,24\mu\ell-2)\ot V_{(\fh,L)},\] {which has a conformal vector}
$$\omega=\omega_\fg\otimes \mathbf{1}+\mathbf{1}\otimes \omega_\fh.$$
We will often view $V_{\wh\fg\rtimes \mv}(\ell,24\mu\ell-2)$ and $V_{(\fh,L)}$ as subalgebras of
$V_{\wh\fg\rtimes \mv}(\ell,24\mu\ell-2)\ot V_{(\fh,L)}$ in the canonical way.
The following result was obtained in \cite[Corollary 5.9]{CLiT}:

\begin{prpt}\label{vahom}
Let $\ell\in \C^\times$.
Then there exists a vertex algebra epimorphism
$$\Theta:V_{\wh{\ft}(\fg)^0}(\gamma_\ell)\rightarrow
V_{\wh\fg\rtimes \mv}(\ell,24\mu\ell-2)\ot V_{(\fh,L)},$$
which is uniquely determined by
\begin{align}\label{theta-1}
&\Theta(t_1^m\ot u)=u\ot e^{m\bfk},  \quad
\Theta(\rk_1)=\ell\bfk,  \quad
\Theta(\rd_1)=\bfd,  \quad
\Theta(\rK_n)=\frac{\ell}{n}e^{n\bfk}, \\
&\quad\qquad \Theta(\rD_n)=n\omega_{-1}e^{n\bfk}-\omega_0\bfd_{-1}e^{n\bfk}
+n^2(\mu\ell-1)\bfk_{-2}e^{n\bfk}\label{theta-2}
\end{align}
for $u\in \fg,\ m\in \Z$ and $n\in \Z^\times$.
\end{prpt}

The following result asserts that every (irreducible) $\phi_{\e}$-coordinated $V_{\wh\fg\rtimes \mv}(\ell,24\mu\ell-2)\ot V_{(\fh,L)}$-module admits
an (irreducible) $\wt\ft(\fg)^\e$-module structure.

\begin{thmt}\label{thm3}
Let $(W,Y_W^\e(\cdot,z))$ be a
$\phi_{\e}$-coordinated $V_{\wh\fg\rtimes \mv}(\ell,24\mu\ell-2)\ot V_{(\fh,L)}$-module with $\ell\in \C^\times$.
Then there is a  $\wt\ft(\fg)^\e$-module structure on $W$ with the actions given by
\begin{align}
t_{0}^{\e-1}\rd_{0}&=  -\te{Res}_{z}z^{-\e}Y_{W}^{\e}(\omega,z),\quad\rk_{0}=\ell,\quad(t_{1}^{m}\ot u)^{\e}(z)=Y_{W}^{\e}(u\ot e^{m\bfk},z),\label{modre1}\\
\rk_{1}^{\e}(z)&=  \ell Y_{W}^{\e}(\bfk,z),\quad\rd_{1}^{\e}(z)=Y_{W}^{\e}(\bfd,z),\quad\rK_{n}^{\e}(z)=\frac{\ell}{n}Y_{W}^{\e}(e^{n\bfk},z),\\
\rD_{n}^{\e}(z)&=  n_{\circ}^{\circ}Y_{W}^{\e}(\omega,z)Y_{W}^{\e}(e^{n\bfk},z)_{\circ}^{\circ}-z^{\e}\frac{d}{dz}\ _{\circ}^{\circ}Y_{W}^{\e}(\bfd,z)Y_{W}^{\e}(e^{n\bfk},z)_{\circ}^{\circ}\label{modre2}\\
 &\  +\frac{1}{2}n\e(\e-1)z^{2\e-2}Y_{W}^{\e}(e^{n\bfk},z)+n^{2}(\mu\ell-1)\ensuremath{z^{\e}\frac{d}{dz}Y_{W}^{\e}(\bfk,z)}Y_{W}^{\e}(e^{n\bfk},z)\nonumber
\end{align}
for $u\in \fg,\ m\in \Z$ and $n\in \Z^\times$.
Furthermore, if the $\phi_{\e}$-coordinated $V_{\wh\fg\rtimes \mv}(\ell,24\mu\ell-2)\ot V_{(\fh,L)}$-module $(W,Y_W^\e(\cdot,z))$ is irreducible, then $W$ is also irreducible as a $\wt\ft(\fg)^\e$-module.
\end{thmt}
	
\begin{proof}	Let $(W,Y_W^\e(\cdot,z))$ be a
$\phi_{\e}$-coordinated $V_{\wh\fg\rtimes \mv}(\ell,24\mu\ell-2)\ot V_{(\fh,L)}$-module with $\ell\ne 0$.
For every $v\in V_{\wh{\ft}(\fg)^0}(\gamma_\ell)$, set
\[Y_W^\e[v,z]:=Y_W^\e(\Theta(v),z)\in \mathrm{Hom}(W,W((z))).\]
 From Proposition \ref{vahom}, it follows that $(W,Y_W^\e[\cdot,z])$
is a $\phi_\e$-coordinated $V_{\wh{\ft}(\fg)^0}(\gamma_\ell)$-module.
Thus, by applying Proposition \ref{prop:phie1}, $W$ becomes a restricted $\wh{\ft}(\fg)_{\mu,\e}$-module of level $\ell$ with
\[a^\e(z)=Y_W^\e[a,z],\quad \forall a\in\mathcal{B}_\fg.\]
Assume further that $(W,Y_W^\e(\cdot,z))$ is irreducible. Then the $\phi_\e$-coordinated $V_{\wh{\ft}(\fg)^0}(\gamma_\ell)$-module
 $(W,Y_W^\e[\cdot,z])$ is also irreducible as the map $\Theta$ is surjective.
This together with Proposition \ref{prop:phie1} gives that  as a $\wh{\ft}(\fg)^\e$-module, $W$ is still irreducible.

For every $v\in V_{\wh\fg\rtimes \mv}(\ell,24\mu\ell-2)\ot V_{(\fh,L)}$,
 it follows  from  Lemma \ref{lem:u0v} and Lemma \ref{lem:actofd}  that
\begin{align*}
[\te{Res}_zz^{-\e} Y_W^\e(\omega,z), Y_W^\e(v,w)]
=Y_W^\e(\omega_0v,w)=Y_W^\e(\mathcal{D}v,w)
=w^\e\frac{d}{dw}Y_W^\e(v,w).
\end{align*}
In particular, for $a\in \mathcal{B}_\fg$, we have
\[[-\te{Res}_zz^{-\e} Y_W^\e(\omega,z), Y_W^\e[a,w]]=-w^\e\frac{d}{dw}Y_W^\e[a,w].\]
This together with Proposition \ref{prop:hgcomm} (12) implies that the $\wh\ft(\fg)^\e$-module $W$ can be extended to a $\wt{\ft}(\fg)^\e$-module with \[t_0^{\e-1}\rd_0=-\te{Res}_zz^{-\e} Y_W^\e(\omega,z).\]
Thus it remains to prove that the operators $Y_W^\e[a,z]$, $a\in \mathcal{B}_\fg$ have the form as given in theorem.
The case that $a\in \wt\fg_1\oplus \sum_{n\in \Z^\times}\C\rK_n$ is obvious, and so we only need to prove that the operators
$Y_W^\e[\rD_n,z]$, $n\in \Z^\times$
coincide with the right hand side of \eqref{modre2}.

Indeed, note that for $i\ge0$ and $n\in\Z$, we have
\[
\bfd_{i}e^{n\bfk}=n\delta_{i,0}e^{n\bfk},\quad\bfk_{i}e^{n\bfk}=0.
\]
This implies that for $j>0$, $(\omega_{\fh})_{j}e^{n\bfk}=(\bfk_{-1}\bfd)_{j}e^{n\bfk}=0$
and hence $\omega_{j}e^{n\bfk}=0$. Thus, by Remark \ref{rem:normalorder}
and Lemma \ref{lem:actofd}, we have
\begin{align*}
Y_{W}^{\e}(\omega_{-1}e^{n\bfk},z) & =_{\circ}^{\circ}Y_{W}^{\e}(\omega,z)Y_{W}^{\e}(e^{n\bfk},z)_{\circ}^{\circ}-\frac{\e}{2}z^{\e-1}Y_{W}(\omega_{0}e^{n\bfk},z)\\
 & =_{\circ}^{\circ}Y_{W}^{\e}(\omega,z)Y_{W}^{\e}(e^{n\bfk},z)_{\circ}^{\circ}-\frac{\e}{2}z^{2\e-1}\frac{d}{dz}Y_{W}(e^{n\bfk},z).
\end{align*}

Similarly, from Remark \ref{rem:normalorder} and Lemma \ref{lem:actofd},
it follows that
\begin{align*}
 & Y_{W}^{\e}\left(\omega_{0}(\bfd_{-1}e^{n\bfk}),z\right)\\
 &\ =z^{\e}\frac{d}{dz}Y_{W}^{\e}(\bfd_{-1}e^{n\bfk},z)\\
 &\ =z^{\e}\frac{d}{dz}\ensuremath{\ _{\circ}^{\circ}Y_{W}^{\e}(\bfd,z)Y_{W}^{\e}(e^{n\bfk},z)_{\circ}^{\circ}-\frac{1}{2}\e z^{\e-1}Y_{W}^{\e}(ne^{n\bfk},z)}\\
 &\ =z^{\e}\frac{d}{dz} \ _{\circ}^{\circ}Y_{W}^{\e}(\bfd,z)Y_{W}^{\e}(e^{n\bfk},z)_{\circ}^{\circ}-\frac{1}{2}n\e(\e-1)z^{2\e-2}Y_{W}^{\e}(e^{n\bfk},z)-\frac{1}{2}n\e z^{2\e-1}\frac{d}{dz}Y_{W}^{\e}(e^{n\bfk},z),
\end{align*}
and that
\begin{align*}
Y_{W}^{\e}(\bfk_{-2}e^{n\bfk},z) &=Y_{W}^{\e}((\omega_{0}\bfk)_{-1}e^{n\bfk},z)=_{\circ}^{\circ}Y_{W}^{\e}(\omega_{0}\bfk,z)Y_{W}^{\e}(e^{n\bfk},z)_{\circ}^{\circ}\\
&=  z^{\e}\frac{d}{dz}Y_{W}^{\e}(\bfk,z)Y_{W}^{\e}(e^{n\bfk},z),
\end{align*}
where we used the fact that $(\omega_{0}\bfk)_{i}e^{n\bfk}$ for $i\ge0$
as $\bfk_{i}e^{n\bfk}=0$.

Finally, by summarizing the above together, we obtain
\begin{align*}
Y_{W}^{\e}[\rD_{n},z] & =Y_{W}^{\e}(\Theta(\rD_{n}),z)\\
 & =Y_{W}^{\e}(n\omega_{-1}e^{n\bfk}-\omega_{0}\bfd_{-1}e^{n\bfk}+n^{2}(\mu\ell-1)\bfk_{-2}e^{n\bfk},z)\\
 & =n\ _{\circ}^{\circ}Y_{W}^{\e}(\omega,z)Y_{W}^{\e}(e^{n\bfk},z)_{\circ}^{\circ}-\frac{1}{2}n\e z^{2\e-1}\frac{d}{dz}Y_{W}^{\e}(e^{n\bfk},z)\\
 & \ -z^{\e}\frac{d}{dz}\ _{\circ}^{\circ}Y_{W}^{\e}(\bfd,z)Y_{W}^{\e}(e^{n\bfk},z)_{\circ}^{\circ}+\frac{1}{2}n\e(\e-1)z^{2\e-2}Y_{W}^{\e}(e^{n\bfk},z)\\
 & \ +\frac{1}{2}n\e z^{2\e-1}\frac{d}{dz}Y_{W}^{\e}(e^{n\bfk},z)+n^{2}(\mu\ell-1)\ensuremath{z^{\e}\frac{d}{dz}Y_{W}^{\e}(\bfk,z)}Y_{W}^{\e}(e^{n\bfk},z)\\
 & =n\ _{\circ}^{\circ}Y_{W}^{\e}(\omega,z)Y_{W}^{\e}(e^{n\bfk},z)_{\circ}^{\circ}-z^{\e}\frac{d}{dz}\ _{\circ}^{\circ}Y_{W}^{\e}(\bfd,z)Y_{W}^{\e}(e^{n\bfk},z)_{\circ}^{\circ}\\
 & \ +\frac{1}{2}n\e(\e-1)z^{2\e-2}Y_{W}^{\e}(e^{n\bfk},z)+n^{2}(\mu\ell-1)\ensuremath{z^{\e}\frac{d}{dz}Y_{W}^{\e}(\bfk,z)}Y_{W}^{\e}(e^{n\bfk},z),
\end{align*}
as desired.
\end{proof}

\textbf{Proof of Theorem \ref{mainth}.} Now we are ready to complete the proof of Theorem \ref{mainth}.
Let $U,\ell,\alpha,\beta\in\C$ be as in Theorem \ref{mainth}.
By Proposition \ref{cocor}, the irreducible highest weight $\wh\fg\rtimes \mv$-module
\[W_1:=L_{\wh\fg\rtimes \mv}(\ell,24\mu\ell-2,U,\beta)\]
is naturally an irreducible $\phi_\e$-coordinated $V_{\wh\fg\rtimes \mv}(\ell,24\mu\ell-2)$-module.
Meanwhile, by Proposition \ref{vhaisvhlmod}, the $\wh\fh$-module
\[W_2:=V_{\wh\fh}(\ell,e^{\alpha\bfk}\C[L])\] is naturally
an irreducible $\phi_\e$-coordinated
$V_{(\fh,L)}$-module.
Thus, by taking tensor product (see Lemma \ref{tensorphiemod}) we obtain an irreducible $\phi_\e$-coordinated
$V_{\wh\fg\rtimes \mv}(\ell,24\mu\ell-2)\ot V_{(\fh,L)}$-module
\[W:=W_1\ot W_2=L_{\wh\fg\rtimes \mv}(\ell,24\mu\ell-2,U,\beta)\ot V_{\wh\fh}(\ell,e^{\alpha\bfk}\C[L])\]
with the actions given by
\begin{align*} &Y_W^\e(u,z)=Y_{W_1}^\e(u,z)\ot 1=u^\e(z)\ot 1=u^\e(z),\\
 &Y_W^\e(\bfk,z)=1\ot Y_{W_2}^\e(\bfk,z)=1\ot \frac{1}{\ell}\bfk^{\e}(z)=\frac{1}{\ell}\bfk^{\e}(z),\\
 &Y_W^\e(\bfd,z)=1\ot Y_{W_2}^\e(\bfd,z)=1\ot \bfd^{\e}(z)=\bfd^\e(z),\\
 &Y_W^\e(e^{n\bfk},z)=1\ot Y_{W_2}^\e(e^{n\bfk},z)=1 \ot E^{n\bfk}(z)=E^{n\bfk}(z),\\
&Y_W^\e(\omega,z)=Y_{W_1}^\e(\omega_\fg,z)\ot 1+
1\ot Y^\e_{W_2}(\omega_\fh,z)=L_\fg^\e(z)\ot 1+1\ot L_\fh^\e(z)=L_\ff^\e(z),\end{align*}
for $ u\in\fg,\ n\in\Z$, where we have used the actions
\eqref{actgvironw}, \eqref{vhlactonvhl}, the fact \eqref{YeWomegafh}, and the conventions \eqref{shortconvention1}, \eqref{shortconvention2}.

According to Theorem \ref{thm3}, the $\phi_\e$-coordinated $V_{\wh\fg\rtimes \mv}(\ell,24\mu\ell-2)\ot V_{(\fh,L)}$-module
$W$
is also  an irreducible $\wt\ft(\fg)^\e$-module with $\rk_0=\ell$ and
\begin{align*}
t_{0}^{\e-1}\rd_{0}&=  -\te{Res}_{z}z^{-\e}Y_{W}^{\e}(\omega,z)=-L_{\ff}(\e-1)+\delta_{\e,1}\mu\ell,\\
(t_{1}^{m}\ot u)^{\e}(z)&=  Y_{W}^{\e}(u\ot e^{m\bfk},z)=(Y_{W_{1}}^{\e}(u,z)\ot1)(1\ot Y_{W_{2}}^{\e}(e^{m\bfk},z))=u^{\e}(z)E^{m\bfk}(z),\\
\rk_{1}^{\e}(z)&=  \ell Y_{W}^{\e}(\bfk,z)=\bfk^{\e}(z),\quad\rd_{1}^{\e}(z)=Y_{W}^{\e}(\bfd,z)=\bfd^{\e}(z),\\
\rK_{n}^{\e}(z)&=  \frac{\ell}{n}Y_{W}^{\e}(e^{n\bfk},z)=\frac{\ell}{n}E^{n\bfk}(z),\\
\rD_{n}^{\e}(z)&=  n\ _{\circ}^{\circ}Y_{W}^{\e}(\omega,z)Y_{W}^{\e}(e^{n\bfk},z)_{\circ}^{\circ}-z^{\e}\frac{d}{dz}\ _{\circ}^{\circ}Y_{W}^{\e}(\bfd,z)Y_{W}^{\e}(e^{n\bfk},z)_{\circ}^{\circ}\\
 &\quad  +\frac{1}{2}n\e(\e-1)z^{2\e-2}Y_{W}^{\e}(e^{n\bfk},z)+n^{2}(\mu\ell-1)\ensuremath{z^{\e}\frac{d}{dz}Y_{W}^{\e}(\bfk,z)}Y_{W}^{\e}(e^{n\bfk},z)\\
&=  n\ _{\circ}^{\circ}L_{\ff}^{\e}(z)E^{n\bfk}(z)_{\circ}^{\circ}+\frac{1}{2}n\e(\e-1)z^{2\e-2}E^{n\bfk}(z)-z^{\e}\frac{d}{dz}\ _{\circ}^{\circ}\bfd^{\e}(z)E^{n\bfk}(z)_{\circ}^{\circ}\\
 &\quad +n^{2}(\mu-\frac{1}{\ell})\ensuremath{z^{\e}\frac{d}{dz}\bfk^{\e}(z)}E^{n\bfk}(z).
\end{align*}
This finishes the proof of Theorem \ref{mainth}.

\section*{Acknowledgement}
This work was supported by the National Natural Science Foundation of China (Nos. 11971396, 11971397, 12131018, 12161141001),
and  the Fundamental Research Funds for the Central Universities (No. 20720200067).

\end{document}